\documentclass[leqno]{article}
\usepackage{authblk}
\usepackage{amsmath,amsthm,amssymb,latexsym,graphicx, mathtools, enumerate,bm,cite,todonotes}

\usepackage[nottoc]{tocbibind}

\usepackage[margin=1in]{geometry}

\usepackage{hyperref}
\hypersetup{
    colorlinks=true,
    allcolors=blue,
}

\newcommand{\EE}{{\mathbb E}}
\newcommand{\NN}{{\mathbb N}}
\newcommand{\PP}{{\mathbb P}}

\newcommand{\RR}{{\mathbb R}}

\newcommand{\cadlag}{}
\def\cadlag/{c\`adl\`ag}

\newcommand{\abs}[1]{ \left| #1 \right|}

\newcommand{\del}{\partial}

\renewcommand{\div}{\on{div}}
\newcommand{\eps}{\varepsilon}

\newcommand{\Id}{\on{Id}}
\newcommand{\ind}{\mathbf{1}}

\newcommand{\loc}{\mathrm{loc}}
\newcommand{\norm}[1]{ \left\| #1 \right\| }

\newcommand{\oline}[1]{\overline{#1}}
\newcommand{\oo}{\infty}
\newcommand{\sgn}{\on{sgn}}
\newcommand{\supp}{\on{supp }}
\newcommand{\Sym}{\mathrm{Sym}}
\newcommand{\tr}{\on{tr}}
\newcommand{\uline}[1]{\underline{#1}}
\newcommand{\w}{\mathrm{w}}
\newcommand{\ws}{{\mathrm{w}\text{-}\star}}

\DeclareMathOperator*{\esssup}{ess\,sup}
\DeclareMathOperator*{\essinf}{ess\,inf}
\DeclareMathOperator*{\argmax}{arg\,max}
\DeclareMathOperator*{\argmin}{arg\,min}

\newcommand{\mcl}{\mathcal}
\newcommand{\mbb}{\mathbb}

\newcommand{\on}{\operatorname}

\newtheorem{lemma}{Lemma}
\newtheorem{proposition}{Proposition}
\newtheorem{theorem}{Theorem}
\newtheorem{corollary}{Corollary}
\newtheorem{remark}{Remark}

\theoremstyle{definition}
\newtheorem{definition}{Definition}

\numberwithin{equation}{section}
\numberwithin{lemma}{section}
\numberwithin{proposition}{section}
\numberwithin{theorem}{section}
\numberwithin{corollary}{section}
\numberwithin{definition}{section}
\numberwithin{remark}{section}

\begin{document}
\title{Transport equations and flows with one-sided Lipschitz velocity fields}
\author{Pierre-Louis Lions}
\affil{Universit\'e Paris-Dauphine and Coll\`ege de France\\
\href{lions@ceremade.dauphine.fr}{\nolinkurl{lions@ceremade.dauphine.fr}}}
\author{Benjamin Seeger
}
\affil{University of Texas at Austin\\ 
\href{mailto:seeger@math.utexas.edu}{\nolinkurl{seeger@math.utexas.edu}}\\
{\normalsize Partially supported by the National Science Foundation under award number DMS-1840314}
}

\maketitle

\begin{abstract}
	We study first- and second-order linear transport equations, as well as ODE and SDE flows, with velocity fields satisfying a one-sided Lipschitz condition. Depending on the time direction, the flows are either compressive or expansive. In the compressive regime, we characterize the stable continuous distributional solutions of both the first and second-order nonconservative transport equations as the unique viscosity solution. Our results in the expansive regime complement the theory of Bouchut, James, and Mancini \cite{BJM}, and we provide a complete theory for both the conservative and nonconservative equations in Lebesgue spaces, as well as proving the existence, uniqueness, and stability of the regular Lagrangian ODE flow. We also provide analogous results in this context for second order equations and SDEs with degenerate noise coefficients that are constant in the spatial variable.
\end{abstract}

\tableofcontents

\section{Introduction}

For a fixed, finite time horizon $T > 0$ and a velocity field $b: [0,T] \times \RR^d \to \RR^d$, we study the linear transport equation
\begin{equation}\label{TE:intro}
	\del_t u + b(t,x) \cdot \nabla u = 0 \quad \text{in } [0,T] \times \RR^d,
\end{equation}
along with the dual, continuity equation
\begin{equation}\label{CE:intro}
	\del_t f + \div( b(t,x) f) = 0 \quad \text{in } [0,T] \times \RR^d, 
\end{equation}
and the associated ordinary differential equation (ODE) flow
\begin{equation}\label{ODE:intro}
	\del_t \phi_{t,s}(x) = b(t,\phi_{t,s}(x)), \quad (s,t,x) \in [0,T] \times [0,T] \times \RR^d, \quad \phi_{s,s} = \Id.
\end{equation}
The goal of the paper is to analyze the three problems, and the relations between them, for vector fields $b$ satisfying the one-sided Lipschitz condition
\begin{equation}\label{intro:monotone}
	\left\{
	\begin{split}
	&(b(t,x) - b(t,y)) \cdot (x-y) \ge -C(t)|x-y|^2 \quad \text{for a.e. } (t,x,y) \in [0,T] \times \RR^d \times \RR^d\\
	&\text{for some nonnegative $C \in L^1([0,T])$}.
	\end{split}
	\right.
\end{equation}

When $b$ is Lipschitz continuous in the space variable, the ODE flow \eqref{ODE:intro} admits a unique global solution, and, through the method of characteristics, \eqref{TE:intro} and \eqref{CE:intro} are uniquely solved for any given smooth initial or terminal data. Moreover, the flow is a diffeomorphism, and therefore the solution operators for either the initial value problem (IVP) or terminal value problem (TVP) for \eqref{TE:intro} and \eqref{CE:intro} are continuous on $L^p_\loc$ for any $p \in [1,\oo]$.

Under the assumption \eqref{intro:monotone}, the time direction plays a nontrivial role, and there is a fundamental difference between the solvability of the flow \eqref{ODE:intro} forward versus backward in time. Indeed, $b$ need not even be continuous, and \eqref{intro:monotone} is equivalent to
\[
	\frac{\nabla b(t,\cdot) + \nabla b(t,\cdot)^T}{2} \ge -C(t)\Id \quad \text{in the sense of distributions.}
\]
In particular, the distribution $\div b$ is a signed measure that is bounded from below, but not in general absolutely continuous with respect to Lebesgue measure. Thus, when $t < s$, the flow \eqref{ODE:intro} is expected to concentrate at sets of Lebesgue measure zero, while the formation of vacuum is witnessed for $t > s$. 

A general study of transport equations and ODEs with irregular velocity fields was initiated by DiPerna and the first author \cite{DL89}, who introduced the notion of renormalized solutions to prove the well-posedness for \eqref{TE:intro} and \eqref{CE:intro} and the almost-everywhere solvability of the flow \eqref{ODE:intro} for $b$ with Sobolev regularity. The DiPerna-Lions theory was extended to equations where only $\mathrm{Sym}(\nabla b) \in L^1$, to Vlasov equations with $BV_\loc$ velocity fields \cite{B_Vlasov_01}, and to two-dimensional problems with a Hamiltonian structure \cite{ABC_09,ABC_13,ABC_14,BD_01,Hauray_03}. Using deep results from geometric measure theory, the renormalization property was extended to the very general case where $b \in BV_\loc$ and $\div b \in L^1$ by Ambrosio \cite{A04}, who also provided a new, measure-theoretic viewpoint on the relationship between uniqueness of nonnegative solutions of \eqref{CE:intro} and the unique solvability of the flow \eqref{ODE:intro} through the idea of superposition. Further developments include equations with velocity fields having a particular structure allowing for less regularity \cite{LbL_04, ChJ_10} and velocity fields belonging to $SBD$ (i.e. $\mathrm{Sym}(\nabla b)$ is a signed measure with no singular Cantor-like part) \cite{ACM_05}. Fine regularity properties of DiPerna-Lions flows were established in \cite{ALM_05, CrDL_08}, and the study of so-called ``nearly incompressible flows'' \cite{ADlM_07} led to the resolution by Bianchini and Bonicatto \cite{BB_20} of Bressan's compactness conjecture \cite{Br_conj_03,Br_app_03}; see also \cite{KMW_21} for related results. For many more details and references, we refer the reader to the surveys \cite{ACetraro, Am_survey_17, Am_Tr_survey_17, Am_Cr_survey_14}.

In the majority of these works, the divergence $\div b$ is assumed to be bounded, or at least absolutely continuous with respect to Lebesgue measure. This is not the case in general for velocity fields satisfying \eqref{intro:monotone}, and so the equations \eqref{TE:intro} and \eqref{CE:intro} do not even have a sense as distributions, because the products $(\div b) u$ and $bf$ are ill-defined for general $u \in L^1_\loc$ or measures $f$. The DiPerna-Lions theory does not, therefore, cover this situation. Moreover, the choice of an appropriate function space of solutions is very sensitive to whether the equations are posed as initial or terminal value problems.

The problems \eqref{TE:intro}-\eqref{ODE:intro} for velocity fields with a one-sided Lipschitz condition have been approached with a variety of methods \cite{CSmfg,Pe_Po_1d_99,Pe_Po_01,Conway_67, PR_97, BJ, BJM}. Our main purpose is to complement these works, and in particular the theory of Bouchut, James, and Mancini \cite{BJM}, by providing complete characterizations of the stable solutions to all three problems in both the compressive and expansive regimes. We also provide some results on the corresponding parabolic equations with a degenerate, second-order term, as well as the SDE analogue of \eqref{ODE:intro} for both the velocity field $b$ and $-b$.

\subsection{Main results}

We relegate a full description of the results, discussions, and examples to the body of the paper. Here, we briefly outline the different sections and the types of results proved within them, and we compare them to the existing literature.

\subsubsection{The compressive regime}

In Section \ref{sec:flow}, we record properties of the backward Filippov flow for \eqref{ODE:intro}, as well as for its Jacobian $J_{t,s}(x) := \det(\nabla \phi_{t,s}(x))$, which is well-defined in $L^\oo$ for a.e. $t \le s$ and $x \in \RR^d$. We employ measure-theoretic arguments to make sense of the right-inverse of the flow in an almost-everywhere sense, as a preliminary step to understanding the forward, regular Lagrangian flow, and prove several properties, the most important of which is its almost-everywhere continuity.

In Section \ref{sec:antimonotone}, we turn to the study of the nonconservative equation\footnote{For a consistent presentation throughout the paper, and in order to emphasize the dual relationship between the two equations, the transport equation \eqref{TE:intro} will always be posed as a terminal value problem, and the continuity equation \eqref{CE:intro} as an initial value problem. The compressive and expansive regimes will be distinguished by the choice of sign in front of the velocity field $b$.}
\begin{equation}\label{ATE:intro}
	\del_t u - b(t,x) \cdot \nabla u = 0 \quad \text{in } (0,T) \times \RR^d, \quad u(T,\cdot) = u_T,
\end{equation}
for which the uniqueness of continuous distributional solutions fails in general. We introduce a new PDE characterization of the ``good'' (stable) solution of \eqref{ATE:intro} as the unique viscosity solution, in the sense of Crandall, Ishii, and the first author \cite{CIL}. This is done by proving a comparison principle for sub and supersolutions. The viscosity solution characterization coincides with the selection of ``good'' solutions by other authors in particular settings \cite{Conway_67, PR_97, Pe_Po_1d_99, Pe_Po_01, BJ, BJM}, allows for robust stability statements, and, moreover, generalizes to the setting of degenerate parabolic problems (see the discussion below).

The ``usual'' viscosity solution theory must be modified due to the lack of global continuity of $b$. In view of the evolution nature of the equations, the $L^1$-dependence in time does not present a problem, and the equations can be treated with the methods of \cite{N90,N92,Idisc85,LP}. To deal with the discontinuity of $b$ in space, sub and supersolutions must be defined with appropriate semicontinuous envelopes of $b$ in the space variable. The direction of the one-sided Lipschitz assumption \eqref{intro:monotone} accounts for the beneficial inequalities in the proof of the comparison principle.

We then introduce further conditions on the velocity field $b$ and terminal data $u_T$ that ensure uniqueness of arbitrary continuous distributional solutions. In particular, the interplay between the regularity of $b$ and $u_T$ plays an important role: if $b \in C^\alpha$ and $u_T \in C^\beta$, then distributional solutions are unique if $\alpha + \beta > 1$, while uniqueness may fail in general if $\alpha + \beta \le 1$, as we show by example.

The latter half of Section \ref{sec:antimonotone} deals with the study of the dual problem to \eqref{ATE:intro}, namely
\begin{equation}\label{ACE:intro}
	\del_t f - \div ( b(t,x)  f) = 0 \quad \text{in }(0,T) \times \RR^d \quad f(0,\cdot) = f_0.
\end{equation}
Even if $f_0 \in L^1_\loc$, the concentrative nature of the flow causes the measure $f(t,\cdot)$ to develop a singular part, and therefore we are led to seek measure-valued solutions. This prevents the duality solution of \eqref{ACE:intro} from being understood in the distributional sense, due to the lack of continuity of $b$. Nevertheless, we prove that, if $b$ is continuous, or if it happens that $f(t,\cdot)$ is absolutely continuous with respect to Lebesgue measure on the time interval $[0,T]$, then the notions of duality and distributional solutions are equivalent.

An important feature of the continuity equation \eqref{ACE:intro} is the failure of renormalization; that is, if $f$ is a duality solution, the measure $|f|$ may fail to be a distributional solution, and may even violate conservation of mass. This is in contrast with the DiPerna-Lions theory, and is a direct consequence of the compressive nature of the backward flow, which can lead to cancellation of the positive and negative parts of $f$. A related phenomenon is the nonuniqueness of distributional solutions of the continuity equation \eqref{CE:intro} with the reverse sign (see below).

\subsubsection{The expansive regime}

In Section \ref{sec:monotone}, we reverse the sign on the velocity field, and study the corresponding problems
\begin{equation}\label{MTE:intro}
	\del_t u + b(t,x) \cdot \nabla u = 0 \quad \text{in } (0,T) \times \RR^d, \quad u(T,\cdot) = u_T
\end{equation}
and
\begin{equation}\label{MCE:intro}
	\del_t f + \div (b(t,x) f) = 0 \quad \text{in }(0,T) \times \RR^d, \quad f(0,\cdot) = f_0.
\end{equation}
In view of the lower bound on the divergence of $b$, we are motivated to seek an $L^p$-based theory for both equations, based on a priori estimates, or equivalently, on the fact that the characteristic flow (the forward ODE \eqref{ODE:intro}) does not concentrate on sets of measure zero.

The initial value problem for the continuity equation \eqref{MCE:intro} was studied in \cite{BJ,BJM}, where a large part of the analysis is based on the fact that locally integrable distributional solutions are \emph{not} unique in general\footnote{Note that in \cite{BJM}, the velocity field $a$ takes the role of $-b$ in \eqref{intro:monotone}, while the time-directions of \eqref{MTE:intro} and \eqref{MCE:intro} are switched, so their setting corresponds to the expansive regime discussed here.}. The same setting is studied in \cite{CSmfg}, where the existence and uniqueness of the forward Filippov flow for \eqref{ODE:intro} is established for a.e. $x \in \RR^d$. 

In the first part of Section \ref{sec:monotone}, we identify a unique ``good'' distributional solution, and prove that the resulting solution operator is continuous on $L^p_\loc$ for all $p \in [1,\oo]$, and stable with respect to regularizations. This coincides with the notion of reversible solution in \cite{BJ,BJM}. 

We then obtain strong stability results for the Bouchut-James-Mancini duality solutions of the nonconservative problem \eqref{MTE:intro} in all $L^p$-spaces, which allow us to prove the renormalization property. Moreover, we introduce a PDE characterization of this duality solution in terms of regularization by $\essinf$- and $\esssup$-convolution. An important ingredient in establishing this characterization is the propagation of almost-everywhere continuity, which, in turn, follows from the renormalization property and the almost-everywhere continuity of the forward flow proved in Section \ref{sec:flow}.

As a consequence of this new characterization, we give a PDE-based proof of the fact that \emph{nonnegative} distributional $L^p$-solutions of \eqref{MCE:intro} are unique, which was established in \cite{CSmfg} using the superposition principle. This result, along with the renormalization property for \eqref{MTE:intro}, allows us to establish the existence, uniqueness, and stability of the forward regular Lagrangian flow for the ODE \eqref{ODE:intro} identified in \cite{CSmfg}. As a byproduct, this also provides a full characterization of the Bouchut-James-Mancini notion of ``good'' (reversible) solution as the pushforward of $f_0$ by the forward flow. Moreover, a distributional solution $f$ is a reversible solution if and only if $|f|$ is also a distributional solution (cf. \cite[Proposition 3.12]{BJM}, which operates under the criterion that $f$ be a so-called ``Jacobian'' solution). 

\subsubsection{SDEs and second order equations}

This paper also contains various results regarding second order versions of \eqref{TE:intro} and \eqref{CE:intro}, as well as stochastic differential equation (SDE) flows. SDEs and degenerate second-order Fokker-Planck equations have been studied from many perspectives, using both the DiPerna-Lions theory and adaptations of the superposition principle, by many authors, including Le Bris and Lions \cite{LbL_08}, Figalli \cite{Fig}, Trevisan \cite{Trevisan_16}, and Champagnat and Jabin \cite{ChJ_SDE_18}; see also the book \cite{LbL_19}. Just as in the first-order setting, the fact that the measure $\div b$ may contain a singular part prevents the application of these theories to the present situation.

In the compressive regime, we extend the viscosity solution theory of Section \ref{sec:antimonotone} to the second order equation
\begin{equation}\label{A2TE:intro}
	-\del_t u + b(t,x) \cdot \nabla u - \tr[ a(t,x) \nabla^2 u] = 0 \quad \text{in }(0,T) \times \RR^d, \quad u(T,\cdot) = u_T,
\end{equation}
where $b$ satisfies the one-sided Lipschitz condition \eqref{intro:monotone} and $a$ is a regular, but possibly degenerate, symmetric matrix. This equation, as well as the dual problem
\begin{equation}\label{A2CE:intro}
	\del_t f - \div (b(t,x) f) - \nabla^2 \cdot (a(t,x) f) = 0 \quad \text{in }(0,T) \times \RR^d, \quad f(0,\cdot) = f_0,
\end{equation}
can be related to the SDE
\begin{equation}\label{backwardSDE:intro}
	d_t \Phi_{t,s}(x) = -b(t,\Phi_{t,s}(x))dt + \sigma(t, \Phi_{t,s}(x))dW_t, \quad t > s, \quad \Phi_{s,s}(x) = x,
\end{equation}
which is the SDE analogue of the backward flow for \eqref{ODE:intro}. Here $W$ is a given Brownian motion and $a =\frac{1}{2} \sigma\sigma^T$. We establish the existence and uniqueness, for every $x \in \RR^d$, of a strong solution in the Filippov sense, and we show that, with probability one, $\Phi_{t,s}$ is H\"older continuous for any exponent less than $1$.

The situation is more complicated in the expansive regime, namely, for the equations
\begin{equation}\label{2TE:intro}
	-\del_t u - b(t,x) \cdot \nabla u - \tr[ a(t,x) \nabla^2 u] = 0 \quad \text{in }(0,T) \times \RR^d, \quad u(T,\cdot) = u_T
\end{equation}
and
\begin{equation}\label{2CE:intro}
	\del_t f + \div (b(t,x) f) - \nabla^2 \cdot (a(t,x) f) = 0 \quad \text{in }(0,T) \times \RR^d, \quad f(0,\cdot) = f_0.
\end{equation}
In the first-order setting, the characterization of the ``good'' distributional solution of the continuity equation \eqref{MCE:intro} relies on the Lipschitz continuity of the backward ODE flow. Adapting similar methods for the second order equation \eqref{2CE:intro} involves establishing Lipschitz continuity of a stochastic flow like \eqref{backwardSDE:intro} with certain time-reversed coefficients (see \eqref{SDE:2MCE} below). While flows of the form \eqref{backwardSDE:intro} are H\"older continuous for any exponent less than $1$, it is an open question as to whether it is Lipschitz with probability one. We relegate a general study of \eqref{2TE:intro} and \eqref{2CE:intro}, and of the stochastic regular Lagrangian flow for
\begin{equation}\label{SDE:intro}
	d_t \Phi_{t,s}(x) = b(t, \Phi_{t,s}(x))dt + \sigma(t, \Phi_{t,s}(x))dW_t, \quad t >s, \quad \Phi_{s,s}(x) = x,
\end{equation}
to future work. The exception\footnote{Another case of interest is when the diffusion matrix $a$ is nondegenerate in which case very general results can be obtained even for locally bounded $b$; see \cite{Fig}.}  is when $\sigma$ is constant in the $\RR^d$-variable. In this case, we prove that a suitable stochastic flow of the form \eqref{backwardSDE:intro} can be inverted, leading, as in the deterministic case, to the existence and uniqueness of a strong solution to \eqref{SDE:intro} for a.e. $x \in \RR^d$, and a corresponding solution theory for the PDEs \eqref{2TE:intro} and \eqref{2CE:intro}.

\subsection{Applications and further study} While interesting in their own right, linear transport equations and ODEs with nonregular velocity fields arise naturally in several equations in fluid dynamics, in which the velocity fields depend nonlinearly on various other physical quantities that are coupled with the transported quantity. Since these equations must be posed a priori in a weak sense, this leads to velocity fields with limited regularity. The DiPerna-Lions and Ambrosio theories have been successfully applied to a number of such situations; see \cite{Ke_Kr_79, AM_DL_KK_03, Am_Bo_DL_04,Cu_Fe_06,ACdPF_12, ACdPF_14,LeF_90,ZM_94}. The one-dimensional Bouchut-James theory of reversible solutions for transport equations with semi-Lipschitz velocity fields has been successfully applied in applications to conservation laws and pressureless gasses; see \cite{BJ_pg_99, BJ_diff_99, GOR_98,GOR_99}.

Nonlinear transport equations also arise in certain models for large population dynamics, specifically mean field games (MFG). In \cite{La_Li_07}, the first author and Lasry introduced a forward-backward system of PDEs modeling a large population of agents in a state of Nash equilibrium. The evolution of the density $f$ of players is described by a continuity equation \eqref{MCE:intro} (or Fokker-Planck equation \eqref{2CE:intro}), where the velocity field $b$ is given by
\begin{equation}\label{MFGb}
	b(t,x) = -\nabla_p H(t, x,\nabla u(t,x)).
\end{equation}
Here, $H$ is a convex Hamiltonian, and $u$ is the solution of the terminal value problem
\begin{equation}\label{HJB}
	-\del_t u - \tr[ a(t,x) \nabla^2 u] + H(t,x, \nabla u(t,x)) = F[f(t,\cdot)] \quad \text{in }(0,T) \times \RR^d,\quad  u(T,\cdot) = G[f(t,\cdot)],
\end{equation}
which is a Hamilton-Jacobi-Bellman equation encoding the optimization problem for a typical agent, and whose influence by the population of agents is described by the coupling functions $F$ and $G$. The velocity field \eqref{MFGb} is the consensus optimal feedback policy of the population of agents at a Nash equilibrium. 

When $a$ is degenerate, or even zero, the function $u$ has limited regularity, and is no better than semiconcave in the spatial variable in general. Therefore, even if $H$ is smooth, the velocity field \eqref{MFGb} may satisfy at most
\begin{equation}\label{open}
	b \in BV_\loc \quad \text{and} \quad (\div b)_- \in L^\oo.
\end{equation}
This falls just outside the DiPerna-Lions-Ambrosio regime, since the measure $(\div b)_+$ may still fail to be absolutely continuous in general. In fact, the well-posedness of a suitable notion of solution for the transport and ODE problems under the general assumptions \eqref{open} remains an open problem.

Many simple but useful MFG models involve a linear-quadratic Hamiltonian of the form
\[
	H(t,x,p) = A(t,x) |p|^2 + B(t,x) \cdot p + C(t,x)
\]
for smooth, real-valued $A,B,C$ with $A > 0$. In this case, it is easy to see that \eqref{MFGb} satisfies the half-Lipschitz condition \eqref{intro:monotone}. This situation was studied by Cardaliaguet and Souganidis \cite{CSmfg} for first-order, stochastic mean field games systems with common noise. In particular, it is proved there that the uniqueness of probability density solutions of \eqref{MTE:intro} gives rise, through the superposition principle, to the uniqueness of optimal trajectories for the probabilistic formulation of the MFG problem, and, moreover, the solution of the stochastic forward-backward system can be used to construct approximate Nash equilibria for the $N$-player game. Our analysis for the Fokker-Planck equation \eqref{2CE:intro} may therefore be expected to yield similar results for stochastic MFG systems with common noise and degenerate, spatially-homogenous, idiosyncratic noise, a special case of the equations considered by Cardaliaguet, Souganidis, and the second author in \cite{CSSmfg}.

The second application of nonlinear transport equations in mean field games is involved with the master equation for a MFG with a finite state space. These equations generally take the form
\begin{equation}\label{FSS}
	\del_t u + b(t,x,u) \cdot \nabla u = c(t,x,u) \quad \text{in } (0,T) \times \RR^d,
\end{equation}
where $u$, $b$, and $c$ all take values in $\RR^d$; coordinate-by-coordinate, \eqref{FSS} is written as
\[
	\del_t u^i + b^j(t,x,u) \del_{x_j} u^i = c^i(t,x,u), \quad i = 1,2,\ldots, d.
\]
Therefore, \eqref{FSS} is a nonconservative hyperbolic system, whose general well-posedness is a difficult question in general; note that, when $d = 1$, \eqref{FSS} becomes a scalar conservation law.

We do not discuss \eqref{FSS} here, but, in a forthcoming paper \cite{LS_increasing}, we study a particular regime of equations taking the form \eqref{FSS}, using a new theory for linear transport equations with velocity fields $b$ that are increasing \emph{coordinate by coordinate}, that is, $\del_{x_j} b^i \ge 0$ for $i \ne j$. 

The extension to infinite dimensions, of both the linear problems \eqref{TE:intro}-\eqref{CE:intro}, as well as the nonlinear equation \eqref{FSS}, remains an interesting question, with numerous applications, including the study of mean field game master equations on the Hilbert space of square-integrable random variables. We aim to study these situations in future work.

\subsection{Notation}
Given a function space $X(\RR^d)$, or $X(\Omega)$ for an appropriate subdomain of $\RR^d$, $X_\loc$ denotes the space of functions (or distributions) $f$ such that $\phi f \in X$ for all $\phi \in C^\oo_c(\RR^d)$. If $X$ is a normed space, the same is not necessarily true for $X_\loc$, but it inherits the topology of local $X$-convergence. For example, $\lim_{n \to \oo }f_n = f$ in $L^p_\loc(\RR^d)$ means that $\lim_{n \to\oo}\norm{f_n - f}_{L^p(B_R)} = 0$ for all $R > 0$. We denote by $L^p_+([0,T])$ the subset of $L^p([0,T])$ consisting of nonnegative functions.

Unless other specified, Banach or Fr\'echet spaces of functions are endowed with the strong topology. For a function space $X$, the subscripts $X_\w$ and $X_{\ws}$ indicate the weak (resp. weak-$\star$) topology.

For $1 \le p < \oo$, $\mcl P_p$ is the space of probability measures $\mu$, with $\int |x|^p \mu(dx) < \oo$, which becomes a complete metric space for the $p$-Wasserstein distance $\mcl W_p$.

The transpose of a matrix $\sigma$ is denoted by $\sigma^T$, and, if $\sigma$ is a square matrix, its symmetric part is denoted by $\mathrm{Sym}(\sigma) := \frac{1}{2}(\sigma + \sigma^T)$. The symbol $\Id$ stands for either the identity map or the identity matrix, the precise meaning being clear from context.

\section{The ODE flow}\label{sec:flow}

This section is focused on the solvability and properties of the flow associated to a velocity field $b$ satisfying\footnote{The linear growth assumption is a standard way to ensure that the a priori estimates for solutions do not blow up. Otherwise, the results of the paper would need a corresponding local theory, as for example in \cite{ACF_15}.}
\begin{equation}\label{A:monotoneB}
	\left\{
	\begin{split}
		&\text{for some } C_0,C_1 \in L^1_+([0,T]) \text{ and for all }t \in [0,T] \text{ and } x,y \in \RR^d,\\
		& |b(t,x)| \le C_0(t)(1 + |x|) \quad \text{and} \\
		& (b(t,x) - b(t,y)) \cdot (x-y) \ge -C_1(t)|x-y|^2.
	\end{split}
	\right.
\end{equation}
Because $b(t,\cdot)$ is not necessarily continuous, the ODE must be interpreted in the Filippov sense \cite{Fil60}, that is, abusing notation, we denote by $b(t,x)$ the convex hull of all limit points of $b(t,y)$ as $y \to x$. For $s \in [0,T]$, we seek absolutely continuous solutions $t \mapsto \phi_{t,s}(x)$ of the problem
\begin{equation}\label{monotone:flow}
	\begin{dcases}
	\del_t \phi_{t,s}(x) \in b(t,\phi_{t,s}(x)), & t \in [0,T],\\
	\phi_{s,s}(x) = x. &
	\end{dcases}
\end{equation}

\begin{remark}\label{R:C0=0}
	If $\dot X(t) \in b(t,X(t))$ and
	\[
		\tilde X(t) := \exp\left( \int_0^t C_1(s)ds\right) X(t) \quad \text{and} \quad \tilde b(t,x) := \exp\left( \int_0^t C_1(s)ds\right) b\left( t, \exp\left( - \int_0^t C_1(s)ds x \right) \right),
	\]
	so that $\dot {\tilde X} \in \tilde b(t, \tilde X(t))$, then $\tilde b$ satisfies \eqref{A:monotoneB} with $C_1 \equiv 0$ and a possibly different $C_0$. In other words, with a change of variables, one may always assume $b$ is monotone without loss of generality.
\end{remark}

We sometimes use the following characterization and properties of half-Lipschitz maps; see \cite[Lemma 2.2]{BJM}. 

\begin{lemma}\label{L:matrixbound}
	A vector field $B: \RR^d \to \RR^d$ satisfies
	\[
		(B(x) - B(y)) \cdot (x-y) \ge -C|x-y|^2 \quad \text{for some } C \ge 0 \text{ and all } x,y \in \RR^d
	\]
	if and only if $\Sym(\nabla B)\ge -C \Id$ in the sense of distributions. We then also have $B \in BD_\loc(\RR^d)$, and
	\[
		\Sym(\nabla B) - (\tr \nabla B) \Id \le (d-1)C \Id.
	\]
\end{lemma}

The space $BD(\RR^d)$ of bounded deformations is the space of vector fields $B: \RR^d \to \RR^d$ such that the symmetric part of the distribution $\nabla B$ is a locally bounded Radon measure, and, as such, is a strictly larger space than $BV(\RR^d)$. For more details, see \cite{Suq}.

We fix a family of regularizations such that
\begin{equation}\label{bregs}
	\left\{
	\begin{split}
		&(b^\eps)_{\eps > 0} \subset L^1([0,T], C^{0,1}(\RR^d)), \quad \lim_{\eps \to 0} b^\eps = b \text{ a.e. in } [0,T] \times \RR^d, \text{ and}\\
		&b^\eps \text{ satisfies \eqref{A:monotoneB} uniformly in $\eps > 0$}.
	\end{split}
	\right.
\end{equation}
For example, we may take $b^\eps(t,\cdot) = b(t,\cdot) * \rho_\eps$ for $\rho_\eps = \eps^{-d} \rho(\cdot/\eps)$, with $\rho \in C^\oo_+(\RR^d)$, $\supp \rho \in B_1$, and $\int \rho = 1$.

\subsection{The backward flow}

We begin the analysis with the backward flow, that is, \eqref{monotone:flow} for $t < s$. This is the time-direction for which the one-sided Lipschitz condition \eqref{A:monotoneB} yields a unique, Lipschitz flow. We record its properties here and refer to \cite{Fil60, Pe_Po_1d_99, Pe_Po_01, Conway_67} for the proofs; see also the work of Dafermos \cite{Daf_77} for the connection to generalized characteristics of conservation laws.

\begin{lemma}\label{L:backwardsflow}
	For every $(s,x) \in [0,T] \times \RR^d$, there exists a unique solution $\phi_{t,s}(x)$ of \eqref{monotone:flow} defined for $(t,x) \in [0,s] \times \RR^d$, satisfying the Lipschitz bound
\begin{equation}\label{LipschitzODE}
	|\phi_{t,s}(x) - \phi_{t,s}(y)| \le \exp\left( \int_t^s C_1(r)dr \right) |x-y| \quad \text{for all } 0 \le t \le s \le T \text{ and } x,y \in \RR^d.
\end{equation}
Moreover, there exists a constant $C > 0$ depending only on $T$ and $C_0$ from \eqref{A:monotoneB} such that
\begin{equation}\label{boundODE}
		|\phi_{t,s}(x)| \le C(|x| + 1) \quad \text{for all } 0 \le t \le s \le T \text{ and } x \in \RR^d,
	\end{equation}
	and
	\begin{equation}\label{timeLipODE}
		\left\{
		\begin{split}
		&|\phi_{t_1,s}(x) - \phi_{t_2,s,x}| \le C (1 + |x|)|t_1 - t_2| \quad \text{and} \\
		&|\phi_{t,s_1}(x) - \phi_{t,s_2}(x)| \le C (1 + |x|)|s_1 - s_2|\\
		&\text{for all } t_1,t_2 \in [0, s], \; s_2,s_2 \in [t, T], \text{ and } x \in \RR^d.
 		\end{split}
		\right.
	\end{equation}
	For all $0 \le r \le s \le t \le T$, $\phi_{r,s} \circ \phi_{s,t} = \phi_{r,t}$. If $(b^\eps)_{\eps > 0}$ are regularizations satisfying \eqref{bregs}, then the corresponding backward flows $\phi^\eps$ converge locally uniformly as $\eps \to 0$ to $\phi$.
\end{lemma}

\begin{remark}\label{R:apriori}
	The a priori local boundedness and time-regularity estimates \eqref{boundODE} and \eqref{timeLipODE}, depending only on $C_0$ and not $C_1$, do not require the half-Lipschitz assumption on $b(t,\cdot)$, and are therefore satisfied for any limiting solutions of the ODE when $b$ satisfies the first condition in \eqref{A:monotoneB}. On the other hand, the half-Lipschitz assumption is crucial for the Lipschitz continuity of the flow \eqref{LipschitzODE}, as well as the uniqueness of the solution. 
\end{remark}
\begin{remark}\label{R:sgnx}
Consider the backward flow in $\RR$ corresponding to $b(t,x) = b(x) = \sgn x$, which is given, for $x \in \RR$ and $s < t$, by
\begin{equation}\label{sgnx}
	\phi_{s,t}(x) = 
	\begin{dcases}
		x+(t-s) & \text{if } x<-t -s,\\
		0 & \text{if } |x| \le t-s, \text{ and}\\
		x-(t-s) & \text{if } x > t-s.
	\end{dcases}
\end{equation}
This demonstrates that, in general, the trajectories of the backward flow may concentrate on sets of measures $0$, in particular, where $b$ has jump discontinuities.

	We will often consider the examples $b(x) = \sgn x$ in subsequent parts of the paper in order to illustrate certain general phenomena and to present counterexamples. Note that, by Remark \ref{R:C0=0}, one can consider similar examples for arbitrary $C_1 \in L_+^1([0,1])$.
\end{remark}

\subsection{The Jacobian for the backward flow}

In view of the Lipschitz regularity \eqref{LipschitzODE}, $\nabla_x \phi_{t,s} \in L^\oo$ for $t \le s$, and so we can define the Jacobian
\begin{equation}\label{backwardsJ}
	J_{t,s}(x) := \det( \nabla_x \phi_{t,s}(x)) \quad \text{for } 0 \le t \le s \le T \quad \text{ and a.e. } x \in \RR^d.
\end{equation}

\begin{lemma}\label{L:backwardsJ}
	Let $J$ be defined as in \eqref{backwardsJ}. Then $J \ge 0$, 
	\begin{equation}\label{Jbound}
		\left\{
		\begin{split}
		&J_{\cdot, s} \in L^\oo([0,s] \times \RR^d) \cap C([0,s], L^1_\loc(\RR^d)) \quad \forall s \in [0,T]
		\quad \text{and} \\
		&J_{t,\cdot} \in L^\oo([t,T] \times \RR^d) \cap C([t,T], L^1_\loc(\RR^d)) \quad \forall t \in [0,T],
		\end{split}
		\right.
	\end{equation}
	\begin{equation}\label{Jestimates}
		\norm{J_{t,s}}_{L^\oo} \le \exp\left( d\int_t^s C_1(r)dr \right) \quad \text{for all } 0 \le t \le s \le T,
	\end{equation}
	and, for all $R > 0$, there exists a modulus of continuity $\omega_R$, which depends on $b$ only through the constants $C_0$ and $C_1$ in \eqref{A:monotoneB}, such that
	\begin{equation}\label{Junifcts}
		\left\{
		\begin{split}
		&\norm{J_{t_1,s} - J_{t_2,s}}_{L^1(B_R)} \le \omega_R(|t_1 - t_2|) \quad \text{for all } t_1,t_2 \in [0,s]
		\quad \text{and} \\
		&\norm{J_{t,s_1} - J_{t,s_2}}_{L^1(B_R)} \le \omega_R(|s_1 - s_2|) \quad \text{for all } s_1,s_2 \in [t,T].
		\end{split}
		\right.
	\end{equation}
	If $(b^\eps)_{\eps > 0}$ are as in \eqref{bregs}, $(\phi^\eps)_{\eps > 0}$ are the corresponding solutions of \eqref{monotone:flow}, and, for $\eps > 0$, $J^\eps = \det(\nabla_x \phi^\eps)$, then
	\begin{equation}\label{Jconv}
		\lim_{\eps \to 0} J^\eps_{\cdot,s} = J_{\cdot,s} \quad \text{weak-$\star$ in } L^\oo([0,s] \times \RR^d)
		\quad \text{and} \quad
		\lim_{\eps \to 0} J^\eps_{t,\cdot} = J_{t,\cdot} \quad \text{weak-$\star$ in } L^\oo([t,T] \times \RR^d).
	\end{equation}
\end{lemma}

\begin{proof}
	It suffices to prove all statements about $J_{\cdot,s}$ on $[0, s]$. The arguments are exactly the same for the other halves using the fact that $s \mapsto \phi_{t,s}$ is the forward flow corresponding to the velocity $-b$. 
	
	The convergence \eqref{Jconv} goes through by compensated compactness arguments for determinants; see the Appendix of \cite{BJM}. The nonnegativity of $J$ now follows, because $J^\eps \ge 0$ for all $\eps$. 
	
	For fixed $\eps > 0$ and $(s,x) \in [0,T] \times \RR^d$, we have 
	\[
		\del_t J^\eps_{t,s}(x) = \div_x b^\eps(t, \phi^\eps_{t,s}(x)) J^\eps_{t,s}(x) \quad \text{for } t \in [0,s].
	\]
	Then \eqref{bregs} implies $\del_t J^\eps_{t,s}(x) \ge -dC_1(t) J^\eps_{t,s}(x)$, and so
	\[
		\frac{\del}{\del t} \left( J^\eps_{t,s}(x) e^{- d\int_t^s C_1(r)dr} \right) \ge 0.
	\]	
	In particular, for $t_1 < t_2 \le s$ and $R > 0$,
	\[
		\int_{B_R} |J^\eps_{t_2,s} - J^\eps_{t_1,s}| \le e^{\int_{t_1}^{t_2} C_1(r)dr} \int_{B_R} J^\eps_{t_2,s} - \int_{B_R} J^\eps_{t_1,s} + \left( e^{\int_{t_1}^{t_2} C_1(r)dr} - 1 \right) \int_{B_R} J^\eps_{t_2,s}.
	\]
	Identifying the modulus of continuity $\omega_R$ in the statement of the Lemma then reduces to proving the uniform-in-$\eps$ continuity of 
	\[
		[0,s] \ni t \mapsto \int_{B_R} J^\eps_{t,s}(x)dx;
	\]
	note that $\int_{B_R} J^\eps_{s,s}(x)dx = |B_R|$, so this will also imply that $\int_{B_R} J^\eps_{t,s}(x)dx$ is bounded uniformly in $\eps$.
	
	In view of the $L^\oo$-boundedness of $J^\eps$, it suffices to prove the uniform-in-$\eps$ continuity in $t$ of $\int  f(x) J^\eps_{t,s}(x)dx$ for any $f \in C_c(\RR^d)$. The change of variables formula gives
	\[
		\int  f(x) J^\eps_{t,s}(x)dx = \int f(\phi^\eps_{s,t}(x))dx.
	\]
	Note that $\del_t \phi^\eps_{s,t}(x) = -b^\eps(t,\phi^\eps_{s,t}(x))$, and the Lipschitz constant in $t$ of $\phi^\eps_{s,t}(x)$ depends only on an upper bound for $|x|$ and the constant $C_0$ in \eqref{A:monotoneB}, and, therefore, is independent of $\eps$.
\end{proof}

When $d = 1$, the $L^\oo$-weak-$\star$ convergence of $J^\eps = \del_x \phi^\eps$ to $J$ can be strengthened via an Aubin-Lions type compactness result.

\begin{proposition}\label{P:Jstrong}
	Assume $d = 1$, and let $J^\eps$ and $J$ be as in Lemma \ref{L:backwardsJ}. Then
	\[
		\lim_{\eps \to 0} J^\eps_{\cdot,s} = J_{\cdot,s} \quad \text{strongly in }L^1_\loc([0,s] \times \RR) \quad \text{and} \quad
		\lim_{\eps \to 0} J^\eps_{t,\cdot} = J_{t,\cdot} \quad \text{strongly in }L^1_\loc([t,T] \times \RR).
	\]
\end{proposition}

\begin{proof}

Fix $t \in [0,T]$ and $R > 0$. Lemma \ref{L:backwardsflow} implies that there exists $M$ independent of $\eps$ such that $|\phi^\eps_{t,s}(x)| \le M$ for all $s \in [t,T]$ and $x \in [-R,R]$. Upon redefining $b$ outside of $[0,T] \times [-2R,2R]$, we find that $\phi_{t,s}(x)$, and therefore $J_{t,s}(x)$, is unchanged, and therefore, in order to prove the $L^1$-convergence in $[t,T] \times [-R,R]$, we may assume without loss of generality that $b$ is bounded uniformly. Applying the transformation $\tilde \phi_{t,s}(x) = \phi_{t,s}(x) - \int_t^s C(r)dr$ for an appropriate $C \in L^1_+([0,T])$ depending on $C_0$ from \eqref{A:monotoneB}, we may also assume $b \ge 1$. 

For $(s,x) \in [t,T] \times \RR$, set $f^\eps(s,x) = J^\eps_{t,s}(x)$. Then $f^\eps$ solves the continuity equation
\[
	\del_s f^\eps + \del_x  \left( b^\eps(s,x) f^\eps \right) = 0 \quad \text{in } [t,T] \times \RR \quad \text{and} \quad f^\eps(t,\cdot) = 1.
\]
For a standard mollifier $\rho \in C^\oo_c([-1,1])$, let $\rho_n = n \rho(\cdot/n)$ and $f^{\eps,n} = \rho_n *_t  f^\eps$ be the mollification of $f^\eps$ only in the time variable. We then have
\[
	\del_s f^{\eps,n} + \del_x \left[ \rho_n *_t  (b^\eps f^\eps)  \right] = 0 \quad \text{in } \left[ t + \frac{1}{n}, T \right] \times \RR
\]
and, for any  $R > 0$,
\[
	\sup_{s \in [t+1/n,T]} \norm{\del_x\left[ \rho_n *_t (b^\eps f^\eps)  \right](s,\cdot)   }_{L^1([-R,R])}  \le
	\sup_{s \in [t+1/n,T]} \norm{ \del_s f^{\eps,n}(s,\cdot) }_{L^1([-R,R])} \le n \norm{\rho'}_{L^1(\RR)} \omega_R\left( \frac{1}{n} \right),
\]
where $\omega_R$ is as in \eqref{Junifcts}. It follows that, for fixed $n \in \NN$, $(\rho_n *_t (b^\eps f^\eps))_{\eps > 0}$ is precompact  in $L^1([t,T] \times [-R,R])$, and so, because
\[
	\lim_{n \to \oo} \rho_n *_t (b^\eps f^\eps)  = bf
\]
in $L^1([t,T] \times [-R,R])$, uniformly in $\eps$, we conclude that $(b^\eps f^\eps)_{\eps > 0}$ is precompact in $L^1([t,T] \times [-R,R])$. This implies that, as $\eps \to 0$, $b^\eps f^\eps$ converges strongly in $L^1([t,T] \times [-R,R])$ to $bf$.

Fix any subsequence $(\eps_n)_{n \ge 0}$ approaching zero as $n \to \oo$. Then there exists a further subsequence such that $f^{\eps_{n_k}} b^{\eps_{n_k}} \xrightarrow{k \to \oo} fb$ almost everywhere, and therefore $f^{\eps_{n_k}} \xrightarrow{k \to \oo} f$ a.e. in $[t,T] \times [-R,R]$ because $b \ge 1$ and 
\begin{align*}
	f^\eps(s,x) - f(s,x) &= \frac{b(s,x)f^\eps(s,x) - b(s,x) f(s,x)}{b(s,x)} \\
	&= \frac{b^\eps(s,x)f^\eps(s,x) - b(s,x) f(s,x)}{b(s,x)} + \frac{\left( b(s,x) - b^\eps(s,x)\right) f^\eps(s,x)}{b(s,x)}.
\end{align*}
The convergence of $f^{\eps_{n_k}}$ to $f$ in $L^1([t,T] \times [-R,R])$, and therefore the convergence of the full family $(f^\eps)_{\eps > 0}$ to $f$, is a consequence of the Lebesgue dominated convergence theorem.
\end{proof}

\begin{remark}\label{R:extendJstrong?}
	The one-dimensional structure is important in the proof of Proposition \ref{P:Jstrong}, in particular, in deducing from the equicontinuity of $J^\eps$ in time that $(b^\eps J^\eps)_{\eps > 0}$ belongs to a precompact subset of $L^1$. It is not immediately clear whether this argument can be extended to multiple dimensions.
\end{remark}

\subsection{The forward flow as the right-inverse of the backward flow}

We next investigate the solvability of \eqref{monotone:flow} forward in time. This is done by analyzing the Jacobian $J$ from the previous subsection in order to invert the backward flow. Similar methods are used in \cite{CSmfg}, and, by including the Jacobian in the analysis, we obtain additionally the almost-everywhere continuity of the inverse.

We will revisit this topic in Section \ref{sec:monotone} when we analyze the forward flow, which will arise from the theory of renormalized solutions of the appropriate transport equation.

\begin{proposition}\label{P:invertbackward}
	For $t \le s$, there exists a set $A_{ts} \subset \RR^d$ of full measure such that, for all $y \in A_{ts}$, $\phi_{t,s}^{-1}(\{y\})$ is a singleton, which we denote by $\left\{ \phi_{s,t}(y) \right\}$. Moreover, there exists a version of the map $\phi_{s,t}:\RR^d \to \RR^d$ such that $\phi_{s,t}$ is continuous a.e.
\end{proposition}

As an intermediate step, we first prove the following.

\begin{lemma}\label{L:topology}
	Assume $0 \le t \le s \le T$ and $K \subset \RR^d$ is nonempty, compact, and connected. Then $\phi_{t,s}^{-1}(K)$ is nonempty, compact, and connected.
\end{lemma}

\begin{proof}
	For $r > 0$, denote $K_r := \bigcup_{y \in K} B_r(y)$. Fix a sequence $(b^n)_{n \in \NN}$ satisfying \eqref{bregs}\footnote{That is, we abuse notation and suppose that $b^n = b^{\eps_n}$ for $(b^\eps)_{\eps >0}$ satisfying \eqref{bregs} and some $(\eps_n)_{n \in \NN}$ satisfying $\lim_{n \to \oo} \eps_n = 0$.}, and let $\phi^n_{t,s}$ denote the corresponding backward flow from the previous subsections.
	
	We first show that
	\begin{equation}\label{upperlimit}
		\phi_{t,s}^{-1}(K) = \bigcap_{r > 0} \bigcup_{n \in\NN} \bigcap_{k \ge n} (\phi_{t,s}^k)^{-1}(K_r).
	\end{equation}
	Suppose $x \in \phi_{t,s}^{-1}(K)$. Then $y = \phi_{t,s}(x) \in K$. Setting $y_n := \phi^n_{t,s}(x)$, we have $\lim_{n \to \oo} y_n = y$ by Lemma \ref{L:backwardsflow}, which means that, for all $r > 0$, there exists $n \in \NN$ such that, for all $k \ge n$, $\phi^k_{t,s}(x) \in B_r(y) \subset K_r$. This proves the $\subset$ direction of \eqref{upperlimit}. 
	
	Now suppose $x$ belongs to the right-hand side of \eqref{upperlimit}. Then, for all $r > 0$, there exists $n \in \NN$ such that $x \in (\phi_{t,s}^k)^{-1}(K_r)$ for all $k \ge n$. Set $y_k := \phi_{t,s}^k(x)$, so that we have $y_k \subset K_r$ for all $k \ge n$. We have $y := \lim_{k \to \oo} y_k = \lim_{k \to \oo} \phi_{t,s}^k(x) = \phi_{t,s}(x)$ by Lemma \ref{L:backwardsflow}. On the other hand, we also have $y \in \oline{K_r}$, and so
	\[
		\phi_{t,s}(x) \subset \bigcap_{r > 0} \oline{K_r} = K.
	\]
	Thus, the $\supset$ direction of \eqref{upperlimit} is established.
	
	The continuity of $\phi_{t,s}$ and the compactness of $K$ imply that $\phi_{t,s}^{-1}(K)$ is closed. We note also that $(\phi^k_{t,s})^{-1} = \phi^k_{s,t}$ satisfies \eqref{boundODE} uniformly in $k$, because the bound only depends on the constant $C_0$ in the linear growth bound of \eqref{A:monotoneB}, which is also satisfied by $-b^k$. This along with \eqref{upperlimit} implies that $\phi_{t,s}^{-1}(K)$ is bounded, and thus compact.
	
	We now show that $\phi_{t,s}$ is surjective. Using again the bound \eqref{boundODE} satisfied uniformly in $k$ for $\phi^k_{s,t}$, we set $x_n := (\phi^n_{t,s})^{-1}(y) = \phi^n_{s,t}(y)$ and note that $(x_n)_{n \in \NN}$ is bounded. Passing to a subsequence, we have $\lim_{k \to \oo} x_{n_k} = x$ for some $x \in \RR^d$, and then $y = \phi^{n_k}_{t,s}(x_k)$, so that $y = \lim_{k \to \oo} \phi^{n_k}_{t,s}(x_k) = \phi_{t,s}(x)$.
	
	Finally, we show $\phi_{t,s}^{-1}(K)$ is connected. For each $k \in \NN$, $(\phi^k_{t,s})^{-1}(K_r)$ is connected, and therefore so is the intersection $\bigcap_{k \ge n} (\phi^k_{t,s})^{-1}(K_r)$ for each $n$. These sets are nested in $n$, so taking the union in $n \in \NN$ yields a connected set. Taking the intersection over $r > 0$ gives the connectedness of $\phi^{-1}_{t,s}(K$).
\end{proof}

\begin{remark}
	The fact that the approximate backward flows converge \emph{uniformly} to $\phi_{t,s}$ is used in the second-to-last paragraph of the proof, in order to show that $\phi_{t,s}$ is surjective.
\end{remark}

\begin{proof}[Proof of Proposition \ref{P:invertbackward}]
	We identify the set by
	\[
		A_{t,s} = \left\{ y \in \RR^d: \text{there exists } x \in \phi_{t,s}^{-1}(\{y\}) \text{ such that $\phi_{t,s}$ is differentiable at $x$ and } J_{t,s}(x) \ne 0 \right\}.
	\]
	We first check that $A_{t,s}$ has full measure. Its complement consists of
	\begin{align*}
		\RR^d \backslash A_{ts} = &\left\{ y \in \RR^d: J_{t,s} = 0 \text{ at the points of differentiability of $\phi_{t,s}$ on } \phi_{t,s}^{-1}(\{y\}) \right\} \\
		&\cup \left\{ y \in \RR^d: \phi_{t,s} \text{ is not differentiable anywhere in } \phi_{t,s}^{-1}(\{y\}) \right\}.
	\end{align*}
	The fact that $\phi_{t,s}$ is differentiable a.e. and the change of variables formula then give
	\[
		|\RR^d \backslash A_{t,s} | = \int_{\RR^d} \ind\{ \phi_{t,s}(x) \in \RR^d \backslash A_{t,s} \} J_{t,s}(x)dx = 0.
	\]
	
	It remains to show that $\phi_{t,s}^{-1}(\{y\})$ is a singleton for all $y \in A_{t,s}$. By Lemma \ref{L:topology}, $\phi_{t,s}^{-1}(\{y\})$ is nonempty, compact, and connected. Suppose $x, \tilde x \in \phi_{t,s}^{-1}(\{y\})$ are such that $J_{t,s}(x) \ne 0$. A Taylor expansion gives
	\[
		y= \phi_{t,s}(\tilde x) = \phi_{t,s}(x) + \nabla_x \phi_{t,s}(x) \cdot (x - \tilde x) + o(|x - \tilde x|) = y + \nabla_x \phi_{t,s}(x)\cdot (x - \tilde x) + o(|x - \tilde x|).
	\]
	The invertibility of $\nabla_x \phi_{t,s}(x)$ then implies that, if $|\tilde x - x|$ is sufficiently small, then $\tilde x = x$, or, in other words, $x$ is an isolated point. But then the connected set $\phi_{t,s}(\{y\})$ must be equal to $\{x\}$, and we call $x = \phi_{s,t}(y)$.
	
	For $y \in A_{t,s}$, we then have $( \phi_{t,s} \circ \phi_{s,t})(y) = y$. Since $\phi_{t,s}^{-1}(\{y\})$ is nonempty for any $y \in \RR^d$, we may define a version of $\phi_{s,t}$ on all of $\RR^d$ by imposing that $\phi_{s,t}(y) \in (\phi_{t,s}^{-1})(\{y\})$ for any $y \in A_{t,s}$. For this version, we have $\phi_{t,s} \circ \phi_{s,t} = \Id$ everywhere on $\RR^d$. Suppose now that $y \in A_{t,s}$ and $\lim_{n \to \oo} y_n = y$ for some sequence $(y_n)_{n \in \NN} \subset \RR^d$. Then
	\[
		\lim_{n \to \oo} (\phi_{t,s} \circ \phi_{s,t})(y_n) = (\phi_{t,s} \circ \phi_{s,t})(y). 
	\]
	We have
	\[
		(\phi_{s,t}(y_n))_{n \in \NN} \subset \bigcup_{n \in \NN} (\phi_{t,s})^{-1}(\{y_n\}),
	\]
	which implies by Lemma \ref{L:topology} that $(\phi_{s,t}(y_n))_{n \in \NN}$ is bounded. Letting $z$ be a limit point of this set, we have, by continuity of the backward flow, that $y = \lim_{n \to \oo} y_n = \phi_{t,s}(z)$, and therefore $z = \phi_{s,t}(y)$.
\end{proof}

\begin{remark}\label{R:fJ?}
We shall see in Section \ref{sec:monotone} that the forward flow is always $BV$ in space. Therefore, the ``forward Jacobian'' $J_{t,s}$ for $t > s$ can only be understood as a measure. Indeed, returning to the example $b(t,x) = \sgn x$ on $\RR$, the right inverse $\phi_{t,s}$ of $\phi_{s,t}$ given by \eqref{sgnx} is $\phi_{t,s}(x) = x + (\sgn x)(t-s)$ for $s \le t$, which is discontinuous only at $0$. The backward Jacobian is given by $J_{s,t}(x) = \ind\left\{ |x| \ge t-s \right\}$, and the forward one is $J_{t,s} = 1 + 2(t-s) \delta_0$.
\end{remark}

\begin{remark}\label{R:leftinv?}
The formula $\phi_{s,t} \circ \phi_{t,s} = \Id$ makes sense a.e. if $s < t$, because $\phi_{s,t}$ is Lipschitz and $\phi_{t,s}$ is measurable. On the other hand, $\phi_{t,s}$ is not also a left-inverse, since the formula $\phi_{t,s} \circ \phi_{s,t}$ does not make sense. In the above example, $\phi_{s,t}(x)$ is equal to $0$, for $|x| \le t-s$, and $0$ is a point of discontinuity for $\phi_{t,s}$. In general, the concentration of $\phi_{s,t}$ on sets of measure $0$ forbids applying $\phi_{t,s}$ as a left-inverse.
\end{remark}

\subsection{Compressive stochastic flows} We now fix a matrix-valued map
\begin{equation}\label{A:Sigma}
	\Sigma \in L^2([0,T], C^{0,1}(\RR^d; \RR^{d \times m})),
\end{equation}
and assume that
\begin{equation}\label{A:BM}
	W: \Omega \times [0,T] \to \RR^m \quad \text{is a standard Brownian motion on a given probability space } (\Omega, \mcl F, \mbb P, \mbb E).
\end{equation}
In order to extend the results in the preceding subsections, and, in particular, to bypass the difficulties of the backward time direction, we consider forward SDEs with drift satisfying the opposite of \eqref{A:monotoneB}, that is,
\begin{equation}\label{A:antimonotoneB}
	B: [0,T] \times \RR^d \to \RR^d, \quad -B \text{ satisfies } \eqref{A:monotoneB},
\end{equation}
and consider the flow
\begin{equation}\label{monotone:SDE}
	\begin{dcases}
	d_s \Phi_{s,t}(x) = B(s,\Phi_{s,t}(x))ds + \Sigma(s,\Phi_{s,t}(x)) dW_s, & s \in [t,T],\\
	\Phi_{t,t}(x) = x.
	\end{dcases}
\end{equation}
Once again, \eqref{monotone:SDE} must be understood in the Filippov sense, which means, for $s \in [t,T]$,
\begin{equation}\label{integraleq}
	\Phi_{s,t}(x) = x + \int_t^s \alpha_r dr + \int_t^s \Sigma(r,\Phi_{r,t}(x))dW_r, \quad \alpha_s \in B(s, \Phi_{s,t}(x)),
\end{equation}
and we remark that our assumptions will allow us to always consider probabilistically strong solutions; that is, we solve \eqref{integraleq} path by path for almost every continuous $W$ with respect to the Wiener measure. Depending on the context in later sections (in particular, the time direction of solvability for the transport and continuity equations), we consider different examples for $B$ and $\Sigma$ for which these assumptions are satisfied.

\begin{lemma}\label{L:backwardsstochflow}
	For every $(t,x) \in [0,T] \times \RR^d$ and $\PP$-almost surely, there exists a unique strong solution $\Phi_{s,t}(x)$ of \eqref{monotone:SDE} defined on $[t,T] \times \RR^d$. Moreover, for all $p \in [2,\oo)$, there exists a constant $C = C_{p} > 0$ depending only on the assumptions \eqref{A:monotoneB} and \eqref{A:Sigma} such that
\begin{equation}\label{ELipschitzSDE}
	\mbb E |\Phi_{s,t}(x) - \Phi_{s,t}(y)|^p \le C |x-y|^p \quad \text{for all } 0 \le t \le s \le T \text{ and } x,y \in \RR^d,
\end{equation}
\begin{equation}\label{EboundSDE}
		\EE|\Phi_{s,t}(x)|^p \le C(|x|^p + 1) \quad \text{for all } -0 \le t \le s \le T \text{ and } x \in \RR^d,
	\end{equation}
	and
	\begin{equation}\label{EtimeregSDE}
		\left\{
		\begin{split}
		&\EE |\Phi_{s_1,t}(x) - \Phi_{s_2,t}(x)|^p  \le C (1 + |x|)|s_1 - s_2|^{p/2}\\
		&\text{for all $t \in [0,T]$, $s_1,s_2 \in [t,T]$, and } x \in \RR^d.
 		\end{split}
		\right.
	\end{equation}
	With probability one, for all $0 \le r \le s \le t \le T$, $\Phi_{t,s} \circ \Phi_{s,r}= \Phi_{t,r}$. If $(b^\eps)_{\eps > 0}$ are regularizations satisfying \eqref{bregs}, then, with probability one, the corresponding stochastic flows $\Phi^\eps$ converge locally uniformly as $\eps \to 0$ to $\Phi$.
\end{lemma}

\begin{proof}
	For $\eps > 0$, let $B^\eps$ be the convolution of $B$ in space by a standard mollifier (so that $b^\eps := -B^\eps$ satisfies \eqref{bregs}), and let $\Phi^\eps_{t,s}$ denote the corresponding stochastic flow. It\^o's formula, the one-sided Lipschitz assumption, and the Lipschitz continuity of $\Sigma$ yield, for any $p \ge 2$ and some $C \in L^1_+([0,T])$,
	\begin{align*}
		\frac{\del}{\del t} \EE |\Phi^\eps_{t,s}(x) - \Phi^\eps_{t,s}(y)|^p
		\le C(t) \EE |\Phi^\eps_{t,s}(x) - \Phi^\eps_{t,s}(y)|^p,
	\end{align*}
	which, along with Gr\"onwall's inequality, leads to the first statement. The other two estimates are proved similarly, with constants independent of $\eps > 0$.
	
	In view of \eqref{ELipschitzSDE} and \eqref{EtimeregSDE}, the Kolmogorov continuity criterion then yields, for any $R > 0$, $p \ge 2$ and $\delta  \in (0,1)$, a constant $C = C_{R,p,\delta} > 0$ such that, for all $s \in [0,T]$, $\lambda \ge 1$ and $\eps > 0$,
	\[
		\PP\left( \sup_{x,y \in B_R} \sup_{r,s \in [s,T]} \frac{ |\Phi^\eps_{t,s}(x) - \Phi^\eps_{r,s}(y) |}{|x-y|^{1-\delta} + |t-s|^{\frac{1}{2}(1-\delta)} } > \lambda \right) \le \frac{C}{\lambda^p}.
	\]
	It follow that the probability measures on $C([s,T] \times \RR^d; \RR^d)$ induced by the random variables $(\Phi^\eps_{\cdot,s})_{\eps > 0}$ are tight with respect to the topology of locally uniform convergence, and therefore converge weakly along a subsequence as $\eps \to 0$ to a probability measure that gives rise to a weak (in the probabilistic sense) solution of \eqref{monotone:SDE}, for which the estimates in the statement of the lemma continue to hold.
	
A similar computation to the one above reveals that, for a fixed probability space and almost every Brownian path $W$, the solution of \eqref{monotone:SDE} is unique. The pathwise uniqueness then implies, by a standard argument due to Yamada and Watanabe \cite{YW}, that there is a unique strong solution for every $x \in \RR^d$.
	\end{proof}
	
\begin{remark}\label{R:Lipstoch?}
	It is an open question whether $\Phi_{t,s}$ is Lipschitz continuous, even if $B$ is Lipschitz. When $B$ is Lipschitz and $\Sigma \in C^{1,\alpha}$ for some $\alpha \in (0,1]$, it turns out the flow $\Phi_{t,s}$ is $C^{1,\alpha'}$ for any $\alpha' \in (0,\alpha)$, but it is not clear how to extend this to the case where $-B$ satisfies the one-sided Lipschitz bound from below. 
	
	As a consequence, an understanding of the Jacobian $\det(\nabla_x\Phi_{t,s}(x))$, or of the stability with respect to regularizations of $B$, is considerably more complicated in the stochastic case. The results of Section \ref{sec:monotone}, where we discuss the expansive regime, are therefore constrained to the first-order case, and we relegate the second-order analysis to future work. One exception is when $\Sigma$ is independent of the spatial variable, in which case a change of variables relates the SDE to an ODE of the form \eqref{monotone:flow} with a random $b$. 
\end{remark}

\subsection{Small noise approximations} We return to the backward flow $\phi_{t,s}$, $0 \le t \le s \le T$, from Lemma \ref{L:backwardsflow}. Recall that the backward flow also corresponds to the forward flow for $-b$; that is, 
\begin{equation}\label{monotone:reversed}
	\frac{\del}{\del s} \phi_{t,s}(x) = - b(s, \phi_{t,s}(x)), \quad s \ge t, \quad \phi_{t,t}(x) = x.
\end{equation}
For $\eps > 0$, let $\phi^\eps_{t,s}(x)$ denote the following stochastic flow
\begin{equation}\label{monotone:epsSDE}
	d_s \phi^\eps_{t,s}(x) = - b(s,\phi^\eps_{t,s}(x))ds + \eps dW_s \quad s \ge t, \quad \phi^\eps_{t,t}(x) = x,
\end{equation}
where $W$ is now a $d$-dimensional Brownian motion. We note that \eqref{monotone:SDE} falls under the assumptions of Lemma \ref{L:backwardsstochflow}, but in fact \eqref{monotone:epsSDE} admits a unique strong solution as soon as $b$ is merely locally bounded \cite{D07, V80}. In general, the limiting solutions as $\eps \to 0$ are not unique; however, we immediately have the following as a consequence of Lemma \ref{L:backwardsstochflow}.

\begin{proposition}\label{P:monotoneepsSDE}
	For every $\eps > 0$, there exists a unique strong solution of \eqref{monotone:epsSDE}. Moreover, as $\eps \to 0$, $\phi^\eps$ converges locally uniformly to $\phi$.
	
	If $J^\eps = \det(\nabla_x \phi^\eps)$, then, as $\eps \to 0$, $J^\eps$ converges weak-$\star$ in $L^\oo([s,T] \times \RR^d)$ and weakly in $C([s,T], L^1_{\loc}(\RR^d))$ to $J$.
\end{proposition}

\section{The compressive regime}\label{sec:antimonotone}

In this section, we consider the transport and continuity equations in the so-called compressive regime. That is, for velocity field $b$ satisfying \eqref{A:monotoneB}, we study the TVP for the nonconservative equation
\begin{equation}\label{E:AMTE}
	-\frac{\del u}{\del t} + b(t,x) \cdot \nabla u = 0 \quad \text{in } (0,T) \times \RR^d \quad \text{and} \quad u(T,\cdot) = u_T \quad \text{in } \RR^d,
\end{equation}
and the IVP for the conservative equation
\begin{equation}\label{E:AMCE}
	\frac{\del f}{\del t} - \div (b(t,x) f) = 0 \quad \text{in } (0,T) \times \RR^d \quad \text{and} \quad f(0,\cdot) = f_0.
\end{equation}
We recall that $\div b$ is bounded from below, and therefore, the direction of time for \eqref{E:AMTE} and \eqref{E:AMCE} does not allow for a solution theory in Lebesgue spaces, due to the concentrative nature of the backward flow analyzed in the previous section. The TVP \eqref{E:AMTE} will be solved in the space of continuous functions, while \eqref{E:AMCE} is solved in the dual space of locally bounded Radon measures. 

We also obtain analogous result for the second-order equations
\begin{equation}\label{E:2AMTE}
	-\frac{\del u}{\del t} - \tr[a(t,x)\nabla^2 u] + b(t,x) \cdot \nabla u = 0 \quad \text{in } (0,T) \times \RR^d \quad \text{and} \quad u(T,\cdot) = u_T
\end{equation}
and
\begin{equation}\label{E:2AMCE}
	\frac{\del f}{\del t} - \div\big[  \div(a(t,x)f) - b(t,x) f \big] = 0 \quad \text{in } (0,T) \times \RR^d \quad \text{and} \quad f(0,\cdot) = f_0,
\end{equation}
where $a = \frac{1}{2} \sigma \sigma^T$ for $\sigma: [0,T] \times \RR^d \to \RR^{d \times m}$ satisfying
\begin{equation}\label{A:sigma}
	\sup_{x \in \RR^d} \frac{ |\sigma(\cdot,x)|}{1 + |x|} + \sup_{y,z \in \RR^d} \frac{|\sigma(\cdot,y) - \sigma(\cdot,z)|}{|y-z|} \in L^2([0,T]).
\end{equation}

\subsection{The nonconservative equation}\label{ss:AMNC}

\subsubsection{Representation formula}

When interpreting \eqref{E:AMTE} in the distributional sense, we are constrained to seek solutions that are continuous. Indeed, the distribution
\[
	b \cdot \nabla u = \div( bu) -( \div b) u
\]
pairs the solution $u$ with $\div b$, which is a measure in general. The other motivating factor is the formal representation formula for the solution of the TVP \eqref{E:AMTE}, which is given in terms of the backward flow:
 \begin{equation}\label{formula:AM}
	u(t,x) = u_T( \phi_{t,T}(x)) \quad \text{for } (t,x) \in [0,T] \times \RR^d.
\end{equation}
This formula and the Lipschitz continuity of $\phi_{t,T}$ given in Lemma \ref{L:backwardsflow} suggest that the solution operator for \eqref{E:AMTE} should preserve continuity. In fact, the formula \eqref{formula:AM} defines a distributional solution, which is uniquely obtained from limits of natural regularizations of the equation.

\begin{theorem}\label{T:AMTE}
	If $u_T \in C(\RR^d)$, then the function $u$ in \eqref{formula:AM} is a distributional solution of \eqref{E:AMTE}. Moreover, if $(b^\eps)_{\eps > 0}$ satisfy \eqref{bregs} and $u^\eps$ is the corresponding solution of \eqref{E:AMTE} with velocity field $b^\eps$, then, as $\eps \to 0$, $u^\eps$ converges locally uniformly to $u$.
\end{theorem}

\begin{proof}
	The unique solution $u^\eps$ for the regularized velocity field is given by $u^\eps(t,\cdot) = u_T \circ \phi^\eps_{t,T}$, where $\phi^\eps$ is the flow corresponding to $b^\eps$. By Lemma \ref{L:backwardsflow}, as $\eps \to 0$, $\phi^\eps$ converges locally uniformly to $\phi$, and so the local-uniform convergence to $u$ follows from the continuity of $u_0$.

	Multiplying the equation for $u^\eps$ by some $\psi \in C^1_c((0,T) \times \RR^d)$ and integrating by parts gives
	\[
		\int_0^T \int_{\RR^d} u^\eps(t,x)\left( \del_t \psi(t,x) - b^\eps(t,x) \cdot \nabla \psi(t,x) + (\div b^\eps(t,x)) \psi(t,x) \right)dxdt = 0.
	\]
	As $\eps \to 0$, $b^\eps \to b$ almost everywhere and $\div b^\eps \rightharpoonup \div b$ weakly in the sense of measures, and so the fact that $u$ is a distributional solution follows.
\end{proof}

Turning next to the second-order equation \eqref{E:2AMTE}, we identify a solution candidate with the appropriate stochastic flow. We do so by changing the time direction in $b$ and $\sigma$ and considering the SDE
\begin{equation}\label{AM:SDE}
	\begin{dcases}
	d_s \Phi_{s,t}(x) = -b(s, \Phi_{s,t}(x))ds + \sigma(s, \Phi_{s,t}(x))dW_s, & s \in [t,T], \\
	\Phi_{t,t} = \Id,
	\end{dcases}
\end{equation}
where $W$ is as in \eqref{A:BM}. Note that \eqref{AM:SDE} is of the type in \eqref{monotone:SDE} and thus falls within the assumptions of Lemma \ref{L:backwardsstochflow}. In particular, if $u_T$ is continuous, then, in view of \eqref{ELipschitzSDE}-\eqref{EtimeregSDE}, the formula
\begin{equation}\label{formula:2AM}
	u(t,x) = \EE[u_T(\Phi_{T,t}(x))]
\end{equation}
defines a continuous function. Moreover, if $u_T$ is Lipschitz, then $u(t,\cdot)$ is Lipschitz for all $t > 0$, and $1/2$-H\"older continuous in time. Note that, in this case, the distribution $\tr[a\nabla^2 u] = \div(a \nabla u) - \div a \cdot \nabla u$ makes sense, because $\nabla u$ and $\div a$ both belong to $L^\oo$.

The following is proved exactly as for Theorem \ref{T:AMTE}, with the use of the estimates in Lemma \ref{L:backwardsstochflow}.

\begin{theorem}\label{T:2AMTE}
	Let $u_T \in C(\RR^d)$ and define $u$ by \eqref{formula:2AM}. If $(b^\eps)_{\eps > 0}$ satisfy \eqref{bregs} and $u^\eps$ is the corresponding solution of \eqref{E:2AMTE} with velocity $b^\eps$, then, as $\eps \to 0$, $u^\eps$ converges locally uniformly to $u$. Moreover, if $u_T \in C^{0,1}$, then 
	\[
		\sup_{(t,x,y) \in [0,T] \times \RR^d \times \RR^d} \frac{|u(t,x) - u(t,y)|}{|x-y|} + \sup_{(r,s,z) \in [0,T] \times \RR^d} \frac{|u(r,z) - u(s,z)|}{|r-s|^{1/2}(1 + |z|) } < \oo,
	\]
	and $u$ is a distributional solution of \eqref{E:2AMTE}.
\end{theorem}

As a special case, we consider, for $\eps > 0$, the ``viscous'' version of \eqref{E:AMTE}, that is
\begin{equation}\label{E:AMVTE}
	-\del_t u^\eps - \frac{\eps^2}{2} \Delta u^\eps + b(t,x) \cdot \nabla u^\eps = 0 \quad \text{in }(0,T) \times \RR^d, \quad u^\eps(T,\cdot) = u_T.
\end{equation}
This uniformly parabolic equation has a unique classical solution for any $u_T \in C(\RR^d)$, which, moreover, is given by $u^\eps(t,x) = \EE[ u_T(\phi^\eps_{t,T}(x))]$, where now $\phi^\eps$ denotes the solution of the SDE \eqref{monotone:epsSDE} from the previous section. Arguing just as in Theorem \ref{T:AMTE} and invoking Proposition \ref{P:monotoneepsSDE} immediately gives the following.

\begin{theorem}\label{T:AMTEvisc}
	As $\eps \to 0$, the solution $u^\eps$ converges locally uniformly to the function $u$ given by \eqref{formula:AM}.
\end{theorem}

\subsubsection{Viscosity solutions}

Although \eqref{formula:AM} and \eqref{formula:2AM} are the distributional solutions that arise uniquely through regularization (either of $b$ or through vanishing viscosity limits), it turns out that distributional solutions are not unique in general (see subsubsection \ref{ss:badvisc} below). It is then a natural question as to whether the ``good'' solutions can be characterized other than as limits of regularizations, or by the explicit formulae. For example, this is done for the one-dimensional problem in \cite{Pe_Po_1d_99} by introducing a sort of entropy condition.

We give a different characterization here using the theory of viscosity solutions \cite{CIL}, which covers both the first- and second-order problems. We present the results here only in the second-order case, which includes the first-order equations when $a = 0$.

We define, for $(t,x,p) \in [0,T] \times \RR^d \times \RR^d$,
\[
	\uline{b}(t,x,p) = \liminf_{z \to x} b(t,z) \cdot p \quad \text{and} \quad \oline{b}(t,x,p) = \limsup_{z \to y} b(t,z) \cdot p.
\]
For fixed $(t,x) \in [0,T] \times \RR^d$, $\uline{b}(t,x,\cdot)$ and $\oline{b}(t,x,\cdot)$ are Lipschitz continuous on $\RR^d$, and, for fixed $(t,p) \in [0,T]$, $\uline{b}(t,\cdot,p)$ and $\oline{b}(t,\cdot,p)$ are respectively lower and upper semicontinuous.

The following definition of viscosity (sup, super) solutions closely resembles the one in \cite{LP}.
\begin{definition}\label{D:visc}
	An upper-semicontinuous (resp. lower-semicontinuous) function $u$ is called a subsolution (resp. supersolution) of \eqref{E:2AMTE} if, for all $\psi: [0,T] \times \RR^d$ that are $C^1$ in $t$ and $C^2$ in $x$, it holds that
	\[
		-\frac{d}{dt} \max_{x \in \RR^d} \left\{ u(t,x)  - \psi(t,x) \right\}
		\le \inf\left\{ \tr[ a(t,y) \nabla^2 \psi(t,y)] - \uline{b}(t,y, \nabla \psi(t,y)) : y \in \argmax\{ u(t,\cdot) - \psi(t,\cdot) \} \right\}
	\]
	(resp. 
	\[
		-\frac{d}{dt} \min_{x \in \RR^d} \left\{ u(t,x)  - \psi(t,x) \right\}
		\ge \sup\left\{ \tr[ a(t,y) \nabla^2 \psi(t,y)] - \oline{b}(t,y, \nabla \psi(t,y)) : y \in \argmin\{ u(t,\cdot) - \psi(t,\cdot) \} \right\} \Big).
	\]
	If $u \in C([0,T] \times \RR^d)$ is both a sub and supersolution, we say $u$ is a solution.
\end{definition}

The comparison principle is proved by doubling the space variable. In particular, we have the following lemma, which follows exactly by methods as in \cite{CImax, N90, N92}. For $(t,x,y) \in [0,T] \times \RR^d \times \RR^d$, we define the nonnegative matrix
\[
	A(t,x,y) := 
	\begin{pmatrix}
		\sigma(t,x) \\
		\sigma(t,y)
	\end{pmatrix}
	\begin{pmatrix}
		\sigma(t,x)^T & \sigma(t,y)^T
	\end{pmatrix}.
\]

\begin{lemma}\label{L:doubled}
	Assume $u$ and $v$ are respectively a sub and supersolution of \eqref{E:2AMTE}. Then $w(t,x,y) = u(t,x) - v(t,y)$ is a subsolution of
	\[
		-\del_t w - \tr[ a(t,x,y) \nabla^2_{(x,y)}w] + \uline{b}(t,x, \nabla_x w) - \oline{b}(t,x, -\nabla_y w) \le 0.
	\]
\end{lemma}

We may now state and prove the comparison principle.

\begin{theorem}\label{T:comparison}
	If $u$ and $v$ are respectively a sub and supersolution of \eqref{E:2AMTE} such that
	\[
		\sup_{(t,x) \in [0,T] \times \RR^d} \frac{ u(t,x)}{1 + |x|} + \sup_{(s,y) \in [0,T] \times \RR^d} \frac{ - v(s,y)}{1 + |y|} < \oo,
	\]
	then $t \mapsto \sup_{x \in \RR^d} \left\{ u(t,x) - v(t,x) \right\}$ is nondecreasing.
\end{theorem}

\begin{proof}
	Define $w(t,x,y) := u(t,x) - v(t,y)$, fix $\delta,\eps > 0$, and define $\Phi_{\delta,\eps}(x,y) = \frac{1}{2 \delta} |x-y|^2 + \frac{1}{2\eps}(|x|^2 + |y|^2)$. In view of the growth of $u$ and $v$ in $x$, for all $t \in [0,T]$, the map $w(t,\cdot,\cdot) - \Phi_{\delta,\eps}(x,y)$ attains a maximum on $\RR^d \times \RR^d$. Moreover, standard arguments from the theory of viscosity solutions (see for instance \cite[Lemma 3.1]{CIL}) imply that there exist $\rho_\delta > 0$ and $\lambda_\eps$ such that $\lim_{\delta\to 0} \rho_\delta^2/\delta = \lim_{\eps \to 0} \eps \lambda_\eps^2 = 0$, and
	\[
		|x - y| \le \rho_\delta \quad \text{and} \quad |x| + |y| \le \lambda_\eps \quad \text{for all } (x,y) \in \argmax \left\{ w(t,\cdot,\cdot) - \Phi_{\delta,\eps} \right\}, \quad t \in [0,T].
	\]
	Therefore, if $t \in [0,T]$ and $(x,y) \in \argmax\left\{ w(t,\cdot,\cdot) - \Phi_{\delta,\eps} \right\}$, we have, for some $C \in L^1_+([0,T])$,
	\begin{align*}
		\tr&[ a(t,x,y) \nabla^2_{(x,y)} \Phi_{\delta,\eps}(x,y) ]\\
		&= \tr\left[
		\left(
		\frac{1}{\delta}
		\begin{pmatrix}
			\Id & -\Id \\
			- \Id & \Id
		\end{pmatrix}
		+ 
		\eps
		\begin{pmatrix}
			\Id & 0 \\
			0  & \Id
		\end{pmatrix}
		\right)
		\begin{pmatrix}
			\sigma(t,x) \\
			\sigma(t,y)
		\end{pmatrix}
		\begin{pmatrix}
			\sigma(t,x)^T & \sigma(t,y)^T
		\end{pmatrix}
		\right]\\
		&\le C(t)\left( \frac{\rho_\delta^2}{\delta} + \eps \lambda_\eps^2 \right)
	\end{align*}
	and
	\begin{align*}
		- \uline{b}&\left(  t,  x, \nabla_x \Phi_{\delta,\eps}(x,y) \right) + \oline{b}\left(  t,  y, - \nabla_y \Phi_{\delta,\eps}(x,y) \right)\\
		&= \limsup_{(z,w) \to ( x,  y)}
		\left\{ - b( t, z) \cdot \left( \frac{ x -  y}{\delta} + \beta  x \right) 
		+ b( t, w) \cdot \left( \frac{ x -  y}{\delta} - \beta  y \right) \right\} \\
		&= \limsup_{(z,w) \to ( x,  y)}
		\left\{ - ( b( t, z) - b( t, w) ) \cdot \frac{z - w}{\delta} - b( t, z) \cdot \beta z + b( t, w) \cdot \beta w \right\} \\
		&\le C(t)\left( \frac{\rho_\delta^2}{\delta} + \eps + \eps \lambda_\eps \right).
	\end{align*}
	It now follows from Definition \ref{D:visc} and Lemma \ref{L:doubled} that, for some $C_{\delta,\eps} \in L^1_+([0,T])$ satisfying $\lim_{(\delta,\eps) \to (0,0)} C_{\delta,\eps} = 0$ in $L^1([0,T])$, 
	\[
		t \mapsto \sup_{(x,y) \in \RR^d \times \RR^d} \left\{ w(t,x,y) - \Phi_{\delta,\eps}(x,y) \right\} - \int_t^T C_{\delta,\eps}(s)ds
	\]
	is nondecreasing. The result follows upon sending $\delta$ and $\eps$ to $0$.
\end{proof}

As a consequence of the comparison theorem, the ``good'' distributional solution of \eqref{E:2AMTE} can be uniquely characterized.

\begin{theorem}
	Assume $u_T \in C(\RR^d)$ and $u_T \cdot (1 + |\cdot|^{-1}) \in L^\oo$. Then \eqref{formula:2AM} is the unique viscosity solution of \eqref{E:2AMTE}.
\end{theorem}

\begin{proof}
	The fact that \eqref{formula:2AM} defines a viscosity solution is due to Theorem \ref{T:2AMTE} and the stability properties of viscosity solutions\footnote{Note that smooth solutions of the equation corresponding to $b^\eps$ satisfying \eqref{bregs}, or of the viscous equation \eqref{E:AMVTE}, are viscosity solutions in the sense of Definition \ref{D:visc}.}. In view of Lemma \ref{L:backwardsstochflow} and the growth of $u_T$, we may appeal to Theorem \ref{T:comparison} to conclude that \eqref{formula:2AM} is the only viscosity solution of the terminal value problem \eqref{E:2AMTE}.
\end{proof}

\subsubsection{(Non)equivalence of distributional and viscosity solutions}\label{ss:badvisc}

For $x \in \RR$, set $b(t,x) = \sgn x$ and $u_T(x) = |x|$. Using the formula \eqref{sgnx} for the backward flow, the solution \eqref{formula:AM} becomes
\begin{equation}\label{goodvisc}
	u(t,x) = (|x| - (T-t))_+.
\end{equation}
However, the Lipschitz function
\begin{equation}\label{badvisc}
	v(t,x) = |x| - (T-t)
\end{equation}
is another distributional solution (and in fact satisfies the equation a.e.). It can also be checked directly that \eqref{badvisc} does not give a viscosity solution of \eqref{E:AMTE}. Indeed, note that $v(t,x) - t$ attains a global minimum at any $(t,0) \in [0,T] \times \RR$. Applying the supersolution definition with $\phi(t,x) = t$ yields the contradictory $-1 \ge 0$.

The uniqueness of distributional solutions fails even if $b$ is continuous. Indeed, if $0 < \alpha < 1$ and $b(t,x) = \sgn x |x|^{\alpha}$ and $u_T(x) = |x|^{1-\alpha}$, then, arguing similarly as in the above example, 
\begin{equation}\label{goodvisc2}
	u(t,x) = \left( |x|^{1-\alpha} - (1-\alpha)(T-t) \right)_+
\end{equation}
and
\begin{equation}\label{badvisc2}
	v(t,x) = |x|^{1-\alpha} - (1-\alpha)(T-t)
\end{equation}
are two distributional solutions, and \eqref{goodvisc2} is the one corresponding to \eqref{formula:AM}. Once again, \eqref{badvisc2} can directly be seen to fail the viscosity supersolution property.

In the first example above, $u_T$ is Lipschitz while $b$ is discontinuous, and, while $b$ is continuous in the second example, we take $u_T$ to be non-Lipschitz. This should be compared with the following sufficient criterion for equivalence.

\begin{theorem}\label{T:CbLipu}
	If $b \in C([0,T] \times \RR^d)$ satisfies \eqref{A:monotoneB} and $u_T \in C^{0,1}(\RR^d)$, then there exists a unique distributional solution $u \in C([0,T] , C^{0,1}(\RR^d))$ given by \eqref{formula:AM}.
\end{theorem}

\begin{proof}
	Let $\rho \in C_c^\oo$ be a standard mollifier and, for $\eps > 0$, set $\rho_\eps(x) = \eps^{-d} \rho(\eps^{-1} x)$. Let $u \in C([0,T],C^{0,1}(\RR^d))$ be a distributional solution of \eqref{E:AMTE} and define $u_\eps = u * \rho_\eps$. Then
	\begin{equation}\label{mollifiedae}
		-\del_t u_\eps + b \cdot \nabla u_\eps = r_\eps \quad \text{in } (0,T) \times \RR^d,
	\end{equation}
	where
	\[
		r_\eps(t,x) = \int_{\RR^d} \left( b(t,y) - b(t,x) \right) \cdot \nabla u(t,y) \rho_\eps(x-y)dy.
	\]
	Note that $r_\eps \in C([0,T] \times \RR^d)$, and $u_\eps$ solves \eqref{mollifiedae} in the sense of viscosity solutions. Moreover, the continuity of $b$ and boundedness of $\nabla u$ imply that $r_\eps \xrightarrow{\eps \to 0} 0$ locally uniformly. Standard stability results from the theory of viscosity solutions then imply that the limit $u$ of $u_\eps$ is the unique viscosity solution of \eqref{E:AMTE}.
\end{proof}

The above result can be extended by studying the interplay between regularity of $b$ and $u$. 

\begin{theorem}\label{T:alphabbetau}
	Suppose that $\alpha,\beta \in (0,1]$ satisfy $\alpha + \beta > 1$, $b$ satisfies \eqref{A:monotoneB} and $\sup_{t \in [0,T]} [b(t,\cdot)]_{C^\alpha} < \oo$, and $u$ is a distributional solution of \eqref{E:AMTE} such that $\sup_{t \in [0,T]} [u(t,\cdot)]_{C^\beta} < \oo$. Then $u$ is the unique viscosity solution of \eqref{E:AMTE}.
\end{theorem}

\begin{remark}
	The condition on $\alpha + \beta$, and, in particular, the strict inequality, is sharp, as the example above with $b(x) = \sgn x |x|^{\alpha}$ and $u_T(x) = |x|^{1-\alpha}$ shows.
\end{remark}

\begin{proof}[Proof of Theorem \ref{T:alphabbetau}]
	Arguing similarly as for Theorem \ref{T:CbLipu}, it suffices to prove that 
	\[
		r_\eps = (b \cdot \nabla u) * \rho_\eps - b \cdot \nabla (u * \rho_\eps) \xrightarrow{\eps \to 0} 0 \quad \text{locally uniformly},
	\]
	where $\rho_\eps$ is a standard mollifier. We note that $r_\eps = M_\eps[ b(t,\cdot), u(t,\cdot)]$, where the bilinear operator $M_\eps$ is defined, for sufficiently regular $(B,U): \RR^d \to \RR^d \times \RR$, by
	\[
		M_\eps[ B,U] = \int_{\RR^d} \left( B(y) - B(x) \right) \cdot \nabla U(y) \rho_\eps(x-y)dy.
	\]
	Standard interpolation arguments give, for some $C >0$ depending on $\alpha$ and $\beta$, for all $(B,U) \in C^\alpha \times C^\beta$,
	\[
		\left| M_\eps[B,U] \right| \le C\eps^{\alpha + \beta - 1} [B]_{C^\alpha} [U]_{C^\beta}.
	\]
	Therefore $|r_\eps(t,x)| \le C [b(t,\cdot)]_{C^\alpha} [u(t,\cdot)]_{C^\beta} \eps^{\alpha + \beta - 1}$, and we conclude upon sending $\eps \to 0$.
\end{proof}

\subsection{The conservative equation}\label{ss:AMCE}

\subsubsection{Duality solutions} 
For either of the two conservative equations \eqref{E:AMCE} and \eqref{E:2AMCE}, the tendency of the backward flow to concentrate on sets of Lebesgue measure zero implies that, even if $f_0$ is absolutely continuous with respect to the Lebesgue measure, $f(t,\cdot)$ may develop a singular part for $t > 0$.

This presents an obstacle in defining solutions in the sense of distributions, since the product of the discontinuous vector field $b$ with a singular measure $f$ may not be well-defined. Instead, we directly define solutions in duality with the nonconservative equation.

\begin{definition}\label{D:dualmeasures}
	A map $f \in C([0,T], \mcl M_{\loc,\w})$ is called a solution of \eqref{E:AMCE} if, for all $t \in [0,T]$ and $g \in C_c(\RR^d)$,
	\[
		\int g(x) f(t, dx) = \int g(\phi_{0,t}(x)) f_0(dx).
	\]
\end{definition}

\begin{remark}\label{R:dualitymeaning}
	For $g \in C_c(\RR^d)$ and $t \in [0,T]$, $(s,x) \mapsto g(\phi_{s,t}(x))$ is the solution of the transport equation \eqref{E:AMTE} in $[0,t] \times \RR^d$ with terminal value $g$ at time $t$, and, hence, $f$ is called the duality solution of \eqref{E:AMCE}. Equivalently, $f(t,\cdot)$ is the pushforward by $\phi_{0,t}$ of the measure $f_0$. When $f_0$ is a probability measure, this means that $f(t,\cdot)$ is the law at time $t$ of the stochastic process $\phi_{0,t}(X_0)$, where $X_0$ is a random variable with law $f_0$.
\end{remark}

\begin{remark}\label{R:dualitytestfns}
	The notion of duality solution can be equivalently formulated in relation to nonconservative equations with a right-hand side\footnote{The theory of viscosity solutions of the terminal value problem for \eqref{E:TErhs}  can be formulated following the theory of the previous subsection with little change.}, that is, for $g \in L^1([0,T], C(\RR^d))$, 
	\begin{equation}\label{E:TErhs}
		-\del_t u + b(t,x) \cdot \nabla u = g(t,x) \quad \text{in } (0,T) \times \RR^d.
	\end{equation}
	With this perspective, although the object $\div(bf)$ does not make sense as a classical distribution, the equation can still be applied to particular singular test functions, namely, solutions of equations like \eqref{E:TErhs}. Then the pairing
	\begin{equation}\label{badpairing}
		\int_{\RR^d} u(T,x)f(T,dx) - \int_{\RR^d} u(0,x)f_0(dx) + \int_0^T \int_{\RR^d}  \underbrace{ \left[- \del_t u(t,x) + b(t,x) \cdot \nabla u(t,x) \right]}_{= g(t,x)}  f(t,dx) = 0
	\end{equation}
	has a sense, because the singular terms collapse into a continuous function, which may be paired with $f(t,\cdot)$.
	\end{remark}

\begin{theorem}\label{T:AMCE}
	There exists a unique duality solution $f$ of \eqref{E:AMCE}. If, for $\eps > 0$, $f^\eps$ is the solution corresponding to $b^\eps$ as in \eqref{bregs}, then, as $\eps \to 0$, $f^\eps$ converges weakly in the sense of measures to $f$. If $1 \le p < \oo$, $f_0, g_0 \in \mcl P_p$, and $f$ and $g$ are the corresponding duality solutions, then, for some $C > 0$ depending on $p$ and the constants in \eqref{A:monotoneB}, $\mcl W_p(f_t, g_t) \le C \mcl W_p(f_0,g_0)$.
\end{theorem}

\begin{proof}
	The existence and uniqueness of duality solutions is a direct consequence of the definition. Moreover, the duality solution identity implies that, for any $R > 0$ and for some $C > 0$ depending on the constants in \eqref{A:monotoneB}, $\norm{f(t,\cdot)}_{TV(B_R)} \le \norm{f_0}_{TV(B_{R+C})}$. For $0 \le s < t \le T$ and $g \in C_c(\RR^d)$, we apply the duality formula with the test function $g \circ \phi_{s,t}$ and obtain the identity
	\[
		\int_{\RR^d}g(x) f(t,dx) = \int_{\RR^d} g(\phi_{s,t}(x)) f(s,dx).
	\]
	Then, by Lemma \ref{L:backwardsflow}, for some modulus of continuity $\omega$ depending on the modulus of continuity for $g$,
	\[
		\abs{\int_{\RR^d} g(x)\left[ f(t,dx) - f(s,dx) \right]} \le \omega(|t-s|) \norm{f_0}_{B_{\supp g + C}},
	\]
	and we conclude that $f \in C([0,T], \mcl M_{\loc, \w})$.
	
	For $R > 0$, define $f_{0,R} := f_0 \ind_{B_R}$, and denote by $f_R$ and $f^\eps_R$ the duality solutions of \eqref{E:AMCE} with respectively $b$ and $b^\eps$ and initial condition $f_{0,R}$. It then suffices to prove that, for fixed $R > 0$ as $\eps \to 0$, $f^\eps_R \rightharpoonup f_R$ in the sense of measures. Then, in view of Lemma \ref{L:backwardsflow}, for any $t \in [0,T]$ and $g \in C_c(\RR^d)$ for sufficiently large support,
	\[
		\int_{\RR^d} g(x) f_R(t,dx) = \int_{B_R} g(\phi_{0,t}(x)) f_0(dx) = \int_{\RR^d} g(\phi_{0,t}(x)) f_0(dx) = \int_{\RR^d} g(x) f(t,dx),
	\]
	and similarly for $f^\eps$.

 	Let then $g \in C_c(\RR^d)$ and $t \in (0,T]$ be fixed, and assume without loss of generality that $f_0$ has compact support in $B_R$ for some $R > 0$. Then, for $\eps > 0$,
	\[
		\int_{\RR^d} g(x) f^\eps(t,dx) = \int g(\phi^\eps_{0,t}(x))f_0(dx).
	\]
	so that $\norm{f^\eps}_{TV} \le \norm{f_0}_{TV}$. Moreover, if $\supp g \subset \RR^d \backslash B_{R+C}$ for some $C > 0$ sufficiently large and independent of $\eps > 0$, again by Lemma \ref{L:backwardsflow},
	\[
		\int_{\RR^d} g(x) f^\eps(t,dx) = 0.
	\]
	We may then take a weakly convergent subsequence of $f^\eps$, with limit point $F \in L^\oo([0,T], \mcl M)$, and, sending $\eps \to 0$, we obtain that $F$ satisfies the duality solution identity, and therefore $F = f$.
	
	Choose $h_1,h_2 \in C_c(\RR^d)$ such that, for all $x,y \in \RR^d$, $h_1(x) + h_2(y) \le |x-y|^p$. Then, if $\gamma$ is any coupling between $f_0$ and $g_0$, we compute, using the duality identity and Lemma \ref{L:backwardsflow},
	\begin{align*}
		\int h_1(x) f(t,dx) + \int h_2(y) g(t,dy) &= \iint \left( h_1(\phi_{0,t}(x)) + h_2(\phi_{0,t}(y)) \right) \gamma(dx,dy)\\
		&\le C \iint |x-y|^p \gamma(dx,dy).
	\end{align*}
	Taking the infimum over such $\gamma$ and supremum over such $h_1,h_2$, and using the dual formulation of the $p$-Wasserstein distance, we arrive at the estimate for the Wasserstein distances.
\end{proof}

\begin{remark}
	The final estimate can also be proved using the characterization of $f$ and $g$ as laws of certain stochastic processes (see Remark \ref{R:dualitymeaning}) and the characterization of the Wasserstein metric in terms of random variables.
\end{remark}

We may repeat the above analysis for the second-order conservative equation \eqref{E:AMCE}, the only difference being the lack of a finite speed of propagation. Therefore, all measures are taken to have finite mass over $\RR^d$. Below, $\Phi_{t,0}$ is the stochastic flow satisfying \eqref{AM:SDE}.

\begin{definition}\label{D:2dualmeasures}
	A map $f \in C([0,T], \mcl M_{\w})$ is called a solution of \eqref{E:2AMCE} if, for all $t \in [0,T]$ and $g \in C_b(\RR^d)$,
	\[
		\int g(x) f(t, dx) = \int \EE[g(\Phi_{t,0}(x))] f_0(dx).
	\]
\end{definition}

\begin{remark}
	Once again, such solutions are called duality solutions because $\EE[ g \circ \Phi_{t,0}]$ is the solution of \eqref{E:2AMTE} with terminal value $g$ at time $t$. If $f_0$ is a probability measure, then $f(t,\cdot)$ is the law of the stochastic process $\Phi_{t,0}(X_0)$, where $X_0$ is a random variable with law $f_0$, independent of the Wiener process $W$.
\end{remark}

The following may be proved exactly as for Theorem \ref{T:AMCE}, now invoking the properties of the stochastic flow described by Lemma \ref{L:backwardsstochflow}.

\begin{theorem}\label{T:2AMCE}
	There exists a unique duality solution $f$ of \eqref{E:2AMCE}. If, for $\eps > 0$, $f^\eps$ is the solution corresponding to $b^\eps$ as in \eqref{bregs}, then, as $\eps \to 0$, $f^\eps$ converges weakly in the sense of measures to $f$. If $1 \le p \le \oo$, $f_0, g_0 \in \mcl P_p$, and $f$ and $g$ are the corresponding duality solutions, then, for some $C > 0$ depending on $p$ and the constants in \eqref{A:monotoneB}, $\mcl W_p(f_t, g_t) \le C \mcl W_p(f_0,g_0)$.
\end{theorem}

\subsubsection{On the failure of renormalization}

In view of the formula \eqref{formula:AM}, it is immediate that (viscosity) solutions of \eqref{E:AMTE} satisfy the renormalization property, that is, if $u$ is a viscosity solution and $\beta: \RR \to \RR$ is smooth, then $\beta \circ u$ is also a solution. This is related to the existence and uniqueness of the Lipschitz backward flow; indeed, note that, coordinate by coordinate, $\phi_{t,T}(x)$ is the unique viscosity solution of \eqref{E:AMTE} with terminal value $x$ at time $T$.

We contrast this with the renormalization property for the forward, conservative problem \eqref{E:AMCE}. If $b$ is smooth, then classical computations show that $f$ is a solution if and only if $|f|$, $f_+$, and $f_-$ are all solutions. Because $f(t,\cdot)$ is the pushforward by $f_0$ of the flow $\phi_{0,t}$, this can be viewed as a generalized form of injectivity for the flow. For general $b$ satisfying \eqref{A:monotoneB}, the backward flow is not only not injective, but concentrates at null sets. We therefore cannot expect renormalization to hold in general.

As a concrete example, take again $b(x) = \sgn x$ on $\RR$, and $f_0 = \frac{1}{2} \delta_{1} - \frac{1}{2} \delta_{-1}$. Then, for $t > 0$, $f(t,\cdot) = \frac{1}{2} \delta_{(1-t)_+} - \frac{1}{2} \delta_{-(1-t)_+}$, which means that $f(t,\cdot) \equiv 0$ for $t \ge 1$. However, the solution $F$ of \eqref{E:AMCE} with $F_0 = |f_0| = \frac{1}{2}\delta_1 + \frac{1}{2} \delta_{-1}$ is equal to $F(t,\cdot) = \frac{1}{2} \delta_{(1-t)_+} + \frac{1}{2} \delta_{-(1-t)_+}$, so that $F(t,\cdot) = \delta_0$ for $t \ge 1$. Thus $F_t \ne |f_t|$ for $t \ge 1$; indeed, $|f_t|$ does not even conserve mass. 

The failure of renormalization holds even if we impose $f_0 \in L^1 \cap L^\oo$. For such $f_0$ and for $b(x) = \sgn x$, we have
\[
	f(t,dx) = \left[ f_0(x+t) \ind\left\{ x > 0\right\} + f_0(x-t) \ind\left\{ x < 0 \right\} \right]dx + \left( \int_{[-t,t]} f_0 \right) d \delta_0(x).
\]
Therefore, renormalization fails whenever $f_0$ is nonzero and odd.

We present one more counterexample to renormalization in which $b \in C$ and $f \in L^1$ (as the previous example shows, even if $f_0 \in L^1$, $f(t,\cdot)$ may not be absolutely continuous with respect to Lebesgue measure due to the concentration of the flow). Take $b(t,x) = 2\sgn x |x|^{1/2}$. The backward flow is given by $\phi_{0,t}(x) = \sgn x (|x|^{1/2} - t)_+^2$ for $(t,x) \in [0,T] \times \RR$. For $f_0 \in L^1$, the duality solution is given by
\[
	f(t,dx) = \left( \int_{[-t^2,t^2]}f_0 \right) \delta_0(dx) + f_0\left( \sgn x(|x|^{1/2} + t)^2 \right) \frac{|x|^{1/2} + t}{|x|^{1/2}} dx.
\]
We then take the odd density $f_0(x) = \sgn x |x|^{1/2} \ind_{[-1,1]}(x)$, and the duality solution takes values in $L^1$:
\begin{equation}\label{suddenL1}
	f(t,x) = \sgn x  \frac{ (|x|^{1/2} + t)^2}{|x|^{1/2}} \ind_{[-(1-t)_+^2, (1-t)_+^2]}(x).
\end{equation}
On the other hand, $|f|$ is not the duality solution, or even a distributional solution, since mass is not conserved. The unique duality solution with initial density $|f_0(x)| = |x|^{1/2} \ind_{[-1,1]}(x)$ in this case is given by
\[
	F(t,dx) = \frac{4t^3}{3} \delta_0(dx) +  \frac{ (|x|^{1/2} + t)^2}{|x|^{1/2}} \ind_{[-(1-t)_+^2, (1-t)_+^2]}(x)dx.
\]

\begin{remark}
	One consequence of the commutator lemma of DiPerna and Lions \cite[Lemma II.1]{DL89} is that, if $f \in L^p$ and $b \in W^{1,q}$ with $\frac{1}{p} + \frac{1}{q} \le 1$, then the renormalization property is satisfied. The previous example therefore indicates that these conditions cannot be weakened in general. Indeed, even though $f_0 \in L^1 \cap L^\oo$, the solution $f(t,\cdot)$ given by \eqref{suddenL1} belongs to $L^p$ only for $p \in [1,2)$ when $t > 0$, and the same is true for $\del_x b$. 
\end{remark}

\subsubsection{Equivalence of duality and distributional solutions}

We finish this section by studying the setting where $bf$ can be understood as a distribution, and, therefore, distributional solutions of \eqref{E:AMCE} can be considered.

\begin{theorem}\label{T:dualdist}
	Assume either that $b$ is continuous, or that $f(t,\cdot) \in L^1_\loc$ for all $t \in [0,T]$. Then $f$ is a distributional solution of \eqref{E:AMCE} if and only if $f$ is the unique duality solution.
\end{theorem}

\begin{proof}
	Suppose $f$ is the unique duality solution. Let $(b^\eps)_{\eps > 0}$ be as in \eqref{bregs} and let $f^\eps$ be the corresponding solution of \eqref{E:AMCE}. For $\phi \in C^1_c((0,T) \times \RR^d)$, integrating by parts yields
	\[
		\iint_{(0,T) \times \RR^d} f^\eps(t,x)\left( -\del_t \phi(t,x) + b^\eps(t,x) \cdot \nabla \phi(t,x) \right)dt dx = 0.
	\]
	In the case that $b \in C$, we may choose regularizations $b^\eps$ that converge locally uniformly to $b$. By Theorem \ref{T:AMCE}, as $\eps \to 0$, $f^\eps$ converges weakly in the sense of measures to $f$, and so we may take $\eps \to 0$ above to obtain
	\[
		\iint_{(0,T) \times \RR^d} f(t,dx)\left( -\del_t \phi(t,x) + b(t,x) \cdot \nabla \phi(t,x) \right)dt = 0.
	\]
	Otherwise, if $f \in L^1_\loc$, it follows that $f^\eps$ converges weakly in $L^1_\loc$, and therefore the same is true for $b^\eps f^\eps$ by the dominated convergence theorem. We may then take $\eps \to 0$ in this case as well.
	
	Assume now that $f$ is an arbitrary distributional solution. We aim to show the duality equality in Definition \ref{D:dualmeasures}, and, by a density argument, it suffices to do so for $g \in C_c(\RR^d) \cap C^{0,1}(\RR^d)$. Let $\rho_\eps$ be a standard mollifier as before and set $f_\eps = f * \rho_\eps$. Then $f_\eps$ satisfies
	\[
		\del_t f_\eps - \div (b f_\eps) = \div r_\eps,
	\]
	where $r_\eps = (bf) * \rho_\eps - b f_\eps$. For $t \in (0,T]$, let $u$ be the unique Lipschitz viscosity solution of the terminal value problem
	\[
		-\del_s u + b \cdot \nabla u = 0 \quad \text{in }(0,T) \times \RR^d, \quad u(t,\cdot) = g.
	\]
	By the theory in subsection \ref{ss:AMNC}, $u(s,x) = g(\phi_{s,t}(x))$ and is Lipschitz continuous with compact support. We then compute
	\[
		\del_s \int f_\eps(s,x) u(s,x)dx
		= - \int r_\eps(s,x) \cdot \nabla u(s,x)dx,
	\]
	so that
	\[
		\int f_\eps(t,x) g(x)dx - \int (f_0 * \rho_\eps)(x) g(\phi_{0,t}(x))dx
		= - \int_0^t \int_{\RR^d} r_\eps(s,x) \cdot \nabla u(s,x)dxds.
	\]
	We may then conclude by proving that $r_\eps \xrightarrow{\eps \to 0} 0$ in $L^1_\loc$.
	
	If $f \in L^1_\loc$, this is immediate because, as $\eps \to 0$, both $(bf) * \rho_\eps$ and $b f_\eps$ converge in $L^1_\loc$ to $bf$. If $b \in C$, then, as $\eps \to 0$, both $(bf) * \rho_\eps$ and $b f_\eps$ converge locally in total variation to $bf$. It follows that $r_\eps$ converges locally in total variation to $0$, but, because $r_\eps \in L^1$ for all $\eps > 0$, the convergence in $L^1_\loc$ is established.
\end{proof}

\begin{remark}
	Even in the context of Theorem \ref{T:dualdist}, the renormalization property can fail. Indeed, this is the case for the final example in the previous subsubsection, where both $b \in C$ and $f \in L^1$.
\end{remark}

\section{The expansive regime}\label{sec:monotone}

We continue our analysis of transport and continuity equations with vector fields $b$ satisfying \eqref{A:monotoneB}, and in this section we study the expansive regime. Reversing the sign appearing in front of the velocity field $b$, the initial value problem for the continuity equation becomes
\begin{equation}\label{E:MCE}
	\del_t f + \div(b(t,x)f) = 0 \quad \text{in } (0,T) \times \RR^d \quad \text{and} \quad f(0,\cdot) = f_0,
\end{equation}
and the corresponding dual terminal value problem for the non-conservative transport equation is
\begin{equation}\label{E:MTE}
	\del_t u + b(t,x) \cdot \nabla u = 0 \quad \text{in } (0,T) \times \RR^d \quad \text{and} \quad u(T,\cdot) = u_T.
\end{equation}

Equivalently, we are studying the time-reversed versions of \eqref{E:AMTE} and \eqref{E:AMCE} (in this case, $b$ is replaced with $b(T-t,\cdot)$). As such, the relevant direction of the flow \eqref{monotone:flow} changes in this context: whereas in the previous section, the compressive, backward flow gave rise to the dual solution spaces $C$ and $\mcl M$, here, the expansive, forward flow allows to develop a theory for both \eqref{E:MCE} and \eqref{E:MTE} in Lebesgue spaces. This can also be seen from formal a priori $L^p$ estimates for \eqref{E:MCE} and \eqref{E:MTE}, which follow immediately from the lower bound on $\div b$.

The regime for these equations matches those studied by Bouchut, James, and Mancini \cite{BJM}, in which emphasis is placed on the fact that distributional solutions $f \in C([0,T], L^\oo_{\ws}(\RR^d))$ of \eqref{E:MCE} are not unique in general. Our approach to these equations is similar, in that we use a particular solution of \eqref{E:MCE} to study, by duality, the transport equation \eqref{E:MTE} and the forward ODE flow to \eqref{monotone:flow}. We extend the results of \cite{BJM} by identifying a ``good'' solution of \eqref{E:MTE} for any $f_0 \in L^p_\loc$, where the continuous solution operator on $L^p$ is stable under regularizations in the weak topology of $C([0,T],L^p_\loc(\RR^d))$. 

The terminal value problem \eqref{E:MTE} is then understood both in the dual sense and through the lens of renormalization theory. It is this theory that allows, as in \cite{DL89}, to make sense of the forward ODE flow \eqref{monotone:SDE} as the right-inverse of the backward flow, completing the program initiated in Section \ref{sec:flow}. As a consequence, we then also obtain the uniqueness of nonnegative distributional solutions of \eqref{E:MCE}, and, by extension, a characterization of the ``good'' solution. 

We finish the section by making some remarks about the second-order analogues of \eqref{E:MCE} and \eqref{E:MTE}. Unlike in the previous section, we do not have a full solution theory for general second-order equations, unless the ellipticity matrix is uniformly positive (the case which has already been covered by Figalli in \cite{Fig}) or is degenerate but independent of the space variable. 

\subsection{The conservative equation}

The starting point for the study of the conservative equation \eqref{E:MCE} is that distributional solutions in the sense of distributions are not unique (see also \cite{BJ}, \cite[Section 6]{BJM}). We revisit the example, when $d = 1$, $b(t,x) = \sgn x$. Then $f(t,x) := \sgn x \ind_{|x| \le t}$ is a nontrivial distributional solution of \eqref{E:MCE} belonging to $L^1 \cap L^\oo$ with $f(0,\cdot) = 0$. The uniqueness can be seen as a consequence of the contractive nature of the backward flow \eqref{monotone:flow}, which allows for positive and negative mass to be ``cancelled'' at time $0$, only to appear immediately for $t > 0$. The same phenomenon is what leads to the failure of renormalization for the contractive regime for the continuity equation in subsection \ref{ss:AMCE}. In either case, we remark that, this particular $b$ belongs to $BV(\RR)$, while $\del_x b$ is not absolutely continuous with respect to Lebesgue measure, and so the condition in the work of Ambrosio \cite{A04} that $\div b \in L^1_\loc$ cannot indeed be weakened in general, if one is to hope for renormalization or uniqueness for the continuity equation.

One strategy is to define solutions of \eqref{E:MCE} by duality with the transport equation \eqref{E:AMTE} from the contractive setting. With the theory of Section \ref{sec:antimonotone}, for $g \in C^{0,1}_c(\RR^d)$, we may define a Lipschitz viscosity solution of the initial value problem
\[
	\del_t v + b(t,x) \cdot \nabla v = 0 \quad \text{in } (0,T) \times \RR^d, \quad v(0,\cdot) = g
\]
(because $\tilde v(t,x) := v(T-t,x)$ solves the corresponding terminal value problem \eqref{E:AMTE} with velocity $\tilde b(t,x) = b(T-t,x)$), and then, formally, for $t > 0$, $\int f(t,x) v(t,x)dx = \int f_0(x) g(x)dx$. 

The main problem with this approach is that duality does not define unique solutions, again due to the concentration effect of the backward flow. Taking once more $b(t,x) = \sgn x$, we have, by \eqref{formula:AM},
\[
	v(t,x) = 
	\begin{dcases}
		g(x - (\sgn x )t), & |x| \ge t, \\
		g(0), & |x| \le t.
	\end{dcases}
\]
Therefore, the duality equality fails to give sufficient information to identify $f$ in the cone $\{|x| \le  t\}$, in which $v$ is always constant, regardless of the initial data $g$. Indeed, the two distributional solutions $f \equiv 0$ and $f(t,x) = \sgn x \ind\{|x| \le t\}$ differ in exactly this cone, in which the Jacobian of the backward flow vanishes. It is exactly this observation that lead to the notion of ``exceptional'' solutions of \eqref{E:AMTE} and the exceptional set in \cite{BJM}.

We instead identify a ``good'' solution operator acting on all $f_0 \in L^p_\loc$, $1 \le p \le \oo$, by extending the solution formula in the smooth case, which depends on the backward flow studied in Section \ref{sec:flow}, as well as the corresponding Jacobian. In particular, the ``good solution'' is distinguished by vanishing whenever the Jacobian does. Our approach differs slightly from that of \cite{BJM}, who work with a general class of ``transport flows'' that generalize the backward ODE flow. One advantage of our analysis is that we can directly appeal to the various topological properties of the backward flow proved in Section \ref{sec:flow}.

\subsubsection{Representation formula} 

If $b$ is Lipschitz, then the solution of \eqref{E:MCE} is given by
\begin{equation}\label{formula:M}
	f(t,x) = f_0(\phi_{0,t}(x)) J_{0,t}(x),
\end{equation}
where $\phi_{0,t}(x)$ is the reverse flow defined in Section \ref{sec:flow} and $J_{0,t}(x) = \det(\nabla_x \phi_{0,t}(x))$ is the corresponding Jacobian. One way to derive this formula is through the Feynman-Kac formula for the reversed time equation
\[
	-\del_t \tilde f + b(T-t,x) \cdot \nabla \tilde f + \div_x b(T-t,x) \tilde f = 0 \quad \text{in }(0,T) \times \RR^d, \quad \tilde f(T,\cdot) = f_0,
\]
which gives 
\begin{equation}\label{M:FC}
	f(t,x) = \tilde f(T-t,x) = f_0(\phi_{0,t}(x)) \exp\left( -\int_0^t \div b(s, \phi_{0,s}(x)ds \right),
\end{equation}
and then $J_{0,t}(x) = \exp\left( -\int_0^s \div b(s, \phi_{0,s}(x)ds \right)$. 

In the general case where $b$ satisfies \eqref{A:monotoneB}, the formula \eqref{formula:M} makes sense for arbitrary $f_0 \in L^p_\loc$, $1 \le p \le \oo$. We may then use the various results in Section \ref{sec:flow} to analyze the stability properties of the solution operator defined by the formula \eqref{formula:AM}. We remark in particular that the stability results of Lemma \ref{L:backwardsJ} depend on the determinant structure of the Jacobian, which is somewhat disguised by the exponential expression in \eqref{M:FC}.

\begin{theorem}\label{T:MCE}
	Let $1 \le p \le \oo$, assume that $f_0 \in L^p_\loc(\RR^d)$, and define $f$ by \eqref{formula:M}. Then $f$ is a distributional solution of \eqref{E:MCE}. If $1 \le p < \oo$, $f \in C([0,T],L^p(\RR^d))$, and if $p = \oo$, $f \in C([0,T],L^\oo_\ws(\RR^d))$. There exists a constant $C > 0$ depending only on the assumptions in \eqref{A:monotoneB} such that, for all $R > 0$,
	\begin{equation}\label{MCE:Lp}
		\norm{f(t,\cdot)}_{L^p(B_R)} \le C\norm{f_0}_{L^p(B_{R+C})}.
	\end{equation}
	If $(b_\eps)_{\eps > 0}$ are as in \eqref{bregs} and $(f_\eps)_{\eps > 0}$ are the corresponding solutions of \eqref{E:MCE}, then, as $\eps \to 0$, $f_\eps$ converges to $f$ weakly in $C([0,T], L^p_\loc(\RR^d))$ if $1 \le p < \oo$, and weak-$\star$ in $L^\oo$ if $p = \oo$.
\end{theorem}

\begin{proof}
	When $p = \oo$, the bound \eqref{MCE:Lp} follows from the $L^\oo$ bounds for the flow and Jacobian in Lemmas \ref{L:backwardsflow} and \ref{L:backwardsJ}. We prove the bound when $p < \oo$ for the solutions $f_\eps$ of the equation with $b_\eps$ as in \eqref{bregs}, for a constant independent of $\eps$, and then the estimate for $f$ follows after proving the weak convergence result.
	
	For a constant $C > 0$ independent of $\eps$, by Lemmas \ref{L:backwardsflow} and \ref{L:backwardsJ}, we have $|J_{0,t}|\le C$ and $|\phi_{0,t}(x)| \le R + C$ for $|x| \le R$. Then
	\begin{align*}
		\int_{B_R} |f^\eps(t,x)|^pdx 
		&= \int_{\RR^d} |f_0(\phi^\eps_{0,t}(x))|^p J^\eps_{0,t}(x)^pdx \\
		&\le \norm{J_{0,t}}_\oo^{p-1} \int_{B_R} |f_0(\phi^\eps_{0,t}(x))|^p J^\eps_{0,t}(x)dx
		\le C \int_{B_{R+C}} |f_0(x)|^pdx.
	\end{align*}
	
	It suffices to prove the weak convergence of $f^\eps$ when $p < \oo$ for $f_0 \in C_c$. In the general case, if $\tilde f_0$ is continuous with compact support and we let $\tilde f^\eps$ be the solution with $b^\eps$ and $\tilde f_0$, we have
	\[
		\norm{f^\eps - \tilde f^\eps}_{C([0,T],L^p(B_R))} \le C \norm{f_0 - \tilde f_0}_{L^p(B_{R+C})},
	\]
	and we may then choose $\tilde f_0$ arbitrarily close to $f_0$ in $L^p_\loc$.
	
	By Lemma \ref{L:backwardsflow}, as $\eps \to 0$, $\phi^\eps \to \phi$ uniformly in $[0,T] \times \RR^d$, and therefore $f_0 \circ \phi^\eps_{0,t})$ converges uniformly to $f_0 \circ \phi_{0,t}$. In view of Lemma \ref{L:backwardsJ}, $f^\eps$ converges weakly in the sense of distributions (and therefore, in the sense of locally bounded Borel measures) to $f$. Since $f^\eps$ is bounded in $L^\oo([0,T], L^p_\loc(\RR^d))$, the convergence is actually weak in $L^\oo([0,T], L^p_\loc(\RR^d))$.
	
	If $p = \oo$, then, in particular, $f^\eps \in C([0,T], L^p_\loc(\RR^d))$ for $p < \oo$, uniformly in $\eps$, and we have the convergence as $\eps \to 0$ in the sense of distributions to $f$. In this case, $f \in L^\oo_\loc([0,T] \times \RR^d)$, and so the convergence is weak-$\star$ in $L^\oo_\loc$.
	
	Given $g \in C^1_c((0,T) \times \RR^d)$, integrating by parts gives
	\[
		\iint_{[0,T] \times \RR^d} f^\eps(t,x)\left[ \del_t \phi(t,x) + b^\eps(t,x) \cdot \nabla \phi(t,x) \right] dxdt = 0.
	\]
	As $\eps \to 0$, the bracketed expression converges a.e. to $\del_t \phi(t,x) + b(t,x) \cdot \nabla \phi(t,x)$, and so converges weakly in $L^q$ for all $1 \le q < \oo$ by the dominated convergence theorem. We may therefore send $\eps \to 0$, using the weak convergence of $f^\eps$, to deduce that $f$ is a distributional solution. This implies in particular that $f \in C([0,T], L^p_\w(\RR^d))$, or $C([0,T], L^\oo_\ws(\RR^d))$ if $p = \oo$.
	
	To show that $f \in C([0,T], L^p(\RR^d))$ when $p < \oo$, we may again consider $f_0 \in C_c(\RR^d))$ without loss of generality. Then $f_0 \circ \phi_{0,\cdot} \in C([0,T] \times \RR^d)$, while $J_{0,\cdot} \in C([0,T], L^1_\loc(\RR^d))$ by Lemma \ref{L:backwardsJ}, and the result follows.	
\end{proof}

\begin{remark}
	We often call the solution $f$ defined through \eqref{formula:M} the ``good'' solution. In view of the stability results of Theorem \ref{T:MCE} above, this solution coincides with the notion of reversible solutions in \cite{BJ, BJM}.
\end{remark}

The following is immediate from the formula \eqref{formula:M}.

\begin{corollary}\label{C:Mrenorm}
	If $f$ is a ``good'' solution of \eqref{E:MCE}, then so is $|f|$.
\end{corollary}

Corollary \ref{C:Mrenorm} is in direct contrast to the continuity equation in the compressive setting of the previous section, where renormalization fails. Its proof depends on the formula for the good solution; indeed, despite the weak stability result in Theorem \ref{T:MCE}, this renormalization property cannot be proved by regularization, since we only have the weak convergence as $\eps \to 0$ of $f_\eps$ to $f$. At present, we do not know whether the convergence is strong in $L^p$. This turns out to be equivalent to the strong convergence in $L^1_\loc$ of the Jacobians, and therefore, in view of Proposition \ref{P:Jstrong}, we have the following when $d = 1$.

\begin{theorem}\label{T:strongconvd=1}
	Assume $d = 1$, $f_0 \in L^p_\loc(\RR)$ for $p < \oo$, $(b^\eps)_{\eps > 0}$ is as in \eqref{bregs}, and $f^\eps$ is the corresponding solution of \eqref{E:MCE}. Then, as $\eps \to 0$, $f^\eps$ converges strongly in $C([0,T] , L^p_\loc(\RR))$ to $f$.
\end{theorem}

\begin{proof}
	Just as in the proof of Theorem \ref{T:MCE}, we may assume without loss of generality that $f_0 \in C_c(\RR)$. In that case, $f^\eps$ is bounded in $L^1$ and $L^\oo$, and so the strong $L^p$ convergence reduces to the strong convergence of $J^\eps_{0,\cdot}$ to $J_{0,\cdot}$ in $L^1_\loc([0,T] \times \RR)$ from Proposition \ref{P:Jstrong}.
\end{proof}

\subsubsection{Vanishing viscosity approximation}

The good solution above also arises from vanishing viscosity limits, that is, the limit as $\eps \to 0$ of solutions of
\begin{equation}\label{E:monotoneeps}
	\del_t f^\eps - \frac{\eps^2}{2} \Delta f^\eps + \div(b(t,x) f^\eps) = 0 \quad \text{in } [0,T] \times \RR^d \quad \text{and} \quad f^\eps(0,\cdot) = f_0,
\end{equation}
which has as its unique solution
\begin{equation}\label{formula:SDE}
	f^\eps(t,x) := \EE[ f_0(\phi^\eps_{0,t}(x))J^\eps_{0,t}(x) ],
\end{equation}
where now $\phi^\eps$ and $J^\eps$ denote respectively the stochastic flow and Jacobian from \eqref{monotone:epsSDE}, corresponding to Proposition \ref{P:monotoneepsSDE}.

The proof of the following result follows from Proposition \ref{P:monotoneepsSDE}, and is proved almost exactly as for Theorem \ref{T:MCE}.

\begin{theorem}
	The function $f^\eps$ defined by \eqref{E:monotoneeps} belongs to $C([0,T],L^p_\loc(\RR^d))$ if $1 \le p < \oo$ and $C([0,T],L^\oo_\ws(\RR^d))$ if $p = \oo$, and, as $\eps \to 0$, $f^\eps$ converges weakly in those spaces to $f$.
\end{theorem}

\subsection{The nonconservative equation}

The next step is the study of the terminal value problem \eqref{E:MTE}. Unlike the transport equation \eqref{E:AMTE} with velocity $-b$, which was solved in the space of continuous functions, we cannot define $L^p$ solutions in the distributional sense, as the product $b \cdot \nabla u = \div(bu) - (\div b) u$ does not make sense when $\div b$ is merely a measure. Instead, we initially characterize solutions by duality with \eqref{E:MCE}, which can be seen as a way of restricting the class of test functions to deal with the singularities in $b$ (see Remark \ref{R:dualitytestfns}).

\subsubsection{$L^p$ and $BV$ estimates}

We will first prove a priori $L^p$ and $BV$ estimates for the solution of \eqref{E:MTE}, assuming all the data and solutions are smooth. The $BV$ estimates in particular are crucial to establishing the \textit{strong} convergence in $L^p$ of regularized solutions to a unique limit, which will be the duality solution, adjoint to the equation \eqref{E:MCE}. The $BV$ estimate appears already in \cite[Lemma 4.4]{BJM}. We present an alternate proof here, which is similar to the one for second-order equations we prove later.

\begin{lemma}\label{L:MTEapriori}
	Assume $b$ is smooth and satisfies \eqref{A:monotoneB}, and let $u$ be a smooth solution of \eqref{E:MTE}. Then, for all $1 \le p \le \oo$, there exist $C = C_{p,R} \in L^1_+([0,T])$ and $C_R > 0$ depending only on the bounds in \eqref{A:monotoneB} such that, for all $0 \le t \le T$,
	\[
		\norm{u(t,\cdot)}_{L^p(B_R)} \le \exp\left(\int_0^t C(s)ds \right) \norm{u_T}_{L^p(B_{R+C})} \quad 		\]
	and
	\[
		\norm{u(t,\cdot)}_{BV(B_R)} \le \exp\left(\int_0^t C(s)ds \right)  \norm{u_T}_{BV(B_{R+C})}.
	\]
\end{lemma}

\begin{proof}
	We assume that $u_T$ has compact support, and, therefore, in view of the finite speed of propagation property, so does $u$. The general result for $L^p_\loc$ and $BV_\loc$ is proved similarly.
		
	The $L^\oo$ bound is a consequence of the maximum principle. For $p < \oo$, we compute
	\[
		\frac{\del}{\del t} \int_{\RR^d} |u(t,x)|^pdx = \int_{\RR^d} \div b(t,x) |u(t,x)|^p dx \ge -C_0(t)d \int_{\RR^d} |u(t,x)|^p dx,
	\]
	and the $L^p$ bound follows from Gr\"onwall's inequality.
	
	Now, for $t \le T$ and $x,z \in \RR^d$, set $w(t,x,z) = \nabla u(t,x) \cdot z$. Then $w$ satisfies
	\[
		-\del_t w + b \cdot \nabla_x w + (z \cdot \nabla)b \cdot \nabla_z w = 0.
	\]
	Since $b$ and $w$ are smooth, the renormalization property holds for this transport equation, and so a simple regularization argument shows, in the sense of distributions,
	\[
		\del_t |w| + b \cdot \nabla_x |w| + (z \cdot \nabla)b \cdot \nabla_z |w| = 0.
	\]
	Define $\phi(z) = e^{-|z|^2}$. Then
	\[
		\iint_{\RR^d \times \RR^d} \phi(z) b(t,x) \cdot \nabla_x |w(t,x,z)|dxdz
		 = - \iint_{\RR^d \times \RR^d} \phi(z) \div b(t,x) |w(t,x,z)|dxdz
	\]
	and
	\begin{align*}
		&\iint_{\RR^d \times \RR^d} \phi(z)(z \cdot \nabla)b(t,x) \cdot \nabla_z|w(t,x,z)|dxdz\\
		&= - \iint_{\RR^d \times \RR^d} \left[ \nabla\phi(z) \cdot (z \cdot \nabla b(t,x)) + \phi(z) \div b(t,x) \right] |w(t,x,z)|dxdz.
	\end{align*}
	Therefore, by Lemma \ref{L:matrixbound},
	\begin{align*}
		\del_t \iint_{\RR^d \times \RR^d} |w(t,x,z)|\phi(z)dxdz
		&= \iint_{\RR^d \times \RR^d} \left[ \nabla \phi(z) \cdot (z \cdot \nabla b(t,x)) + 2 \phi(z) \div b(t,x) \right]|w(t,x,z)|dxdz \\
		&=\iint_{\RR^d \times \RR^d} 2e^{-|z|^2} \left[ \div b(t,x) - \nabla b(t,x)z \cdot z\right]dxdz\\
		&\ge -2(d-1)C_0(t) \iint_{\RR^d \times \RR^d} e^{-|z|^2} |w(t,x,z)|dxdz.
	\end{align*}
	The result follows from Gr\"onwall's lemma and the fact that
	\[
		\iint_{\RR^d \times \RR^d} e^{-|z|^2} |w(t,x,z)|dxdz = c_0 \int_{\RR^d} |\nabla u(t,x)|dx,
	\]
	where the constant $c_0 = \int_{\RR^d} e^{-|z|^2} |\nu \cdot z|dz$ is independent of the choice of $|\nu| = 1$ by rotational invariance.
\end{proof}

\subsubsection{Duality solutions}

Proceeding by duality with the conservative forward equation, and using the $BV$-estimates above, then gives the following.

\begin{theorem}\label{T:MTE}
	Assume $1 \le p \le \oo$ and $u_T \in L^p_\loc$. Then there exists a unique function $u \in C([0,T], L^p_\loc(\RR^d))$ (or in $C([0,T], L^\oo_\ws(\RR^d))$ if $p = \oo$) such that, if $(b^\eps)_{\eps > 0}$ is as in \eqref{bregs} and $u^\eps$ denotes the corresponding solution of \eqref{E:MTE}, then, as $\eps \to 0$, $u^\eps$ converges strongly in $C([0,T], L^p(\RR^d))$ for $p < \oo$ and weak-$\star$ in $L^\oo$ to $u$. Moreover, the solution map $u_T \mapsto u$ is linear, order-preserving, and continuous in $L^p_\loc(\RR^d))$. 
	If $s \in [0,T)$, $f_s \in L^{p'}(\RR^d)$ and $f \in C([s,T], L^{p'}(\RR^d))$ (or $C([s,T], L^\oo_\ws(\RR^d))$ if $p = 1$) is the good solution of \eqref{E:MCE} with initial data $f(s,\cdot) = f_s$, then
	\[
		\int_{\RR^d} u(s,x)f_s(x)dx = \int_{\RR^d} u_T(x) f(T,x)dx.
	\]
\end{theorem}

\begin{remark}
	The function $u$ corresponds with the notion of duality solution presented in \cite{BJM} whenever $u_T$ (and therefore $u(t,\cdot)$ for $t < T$) belongs to $BV_\loc$.
\end{remark}

\begin{proof}
	By Lemma \ref{L:MTEapriori}, $(u^\eps)_{\eps > 0}$ is bounded uniformly in $C([0,T], L^p_\loc(\RR^d))$, and so, along a subsequence, converges weakly as $\eps \to 0$ to some $u$ satisfying the same bounds. 
	
	In order to see that the convergence is strong, note that it suffices, by the $L^p$-boundedness of solution operator implied by Lemma \ref{L:MTEapriori}, to assume that $u_T\in C_c(\RR^d)$. We then have $u^\eps$ bounded in $L^\oo([0,T], BV(\RR^d))$ independently of $\eps$. The identity  $\del_t u^\eps = - b^\eps \cdot \nabla u^\eps$
	then implies that, for any $t_1 < t_2 \le T$ and $R > 0$,
	\[
		\norm{u^\eps(t_1,\cdot) - u^\eps(t_2,\cdot)}_{L^1(B_R)} \le \norm{b}_{L^\oo(B_R)} \sup_{t \in [0,T]} \norm{\nabla u^\eps}_{L^1(B_R)} |t_1 - t_2|.
	\]
	This, along with the uniform $BV$ estimates, implies that $(u^\eps)_{\eps > 0}$ is precompact in $C([0,T], L^1_\loc(\RR^d))$, and, because of the uniform $L^\oo$-bound, precompact in $C([0,T], L^p_\loc(\RR^d))$ for any $p \in [1,\oo)$. It therefore follows that any weakly convergent subsequence actually converges strongly.
	
	If $f^\eps$ is the solution of \eqref{E:MCE} with $f^\eps(s,\cdot) = f_s$, then classical computations involving integration by parts give
	\[
		\int_{\RR^d} u^\eps(s,x)f_s(x)dx = \int_{\RR^d} u_T(x) f^\eps(T,x)dx.
	\]
	Sending $\eps \to 0$ along a subsequence and using the weak convergence of $f^\eps$ and strong convergence of $u^\eps$ shows that any limit point $u$ must satisfy the duality identity with $f$, and is therefore unique. We conclude that the full sequence converges strongly. As before, when $p = \oo$, we obtain the same result since then also $u \in C([0,T], L^p_\loc(\RR^d))$ for any $p < \oo$.
	
\end{proof}

\begin{remark}
	If $u_T \in BV_\loc$, then the duality solution $u$ of \eqref{E:MTE} satisfies $\nabla u \in L^\oo([0,T], \mcl M_\loc(\RR^d))$. Note, however, that this is still not enough to make sense of $u$ as a distributional solution, unless $b$ is continuous.
\end{remark}

\subsubsection{Renormalization}

In Section \ref{sec:antimonotone}, the renormalization property for solutions of the transport equation \eqref{E:AMTE} followed from the formula \eqref{formula:AM}. We prove a similar renormalization property for the transport equation \eqref{E:MTE} in the expansive regime. Here, it depends on the strong convergence in $L^p$ of regularizations.

\begin{theorem}\label{T:MTErenorm}
	Let $1 \le p \le \oo$ and $u_T \in L^p_\loc(\RR^d)$, and let $u \in C([0,T], L^p_\loc(\RR^d))$ be the duality solution of \eqref{E:MTE}. Assume $\beta: \RR \to \RR$ is smooth and satisfies $|\beta(r)| \le (1 + |r|^\alpha)$ for some $\alpha > 0$. Then $\beta \circ u = C([0,T],L^{p/\alpha}_\loc(\RR^d))$ is the duality solution of \eqref{E:MTE} with terminal value $\beta(u(T,\cdot)) = \beta \circ u_T$.
\end{theorem}

\begin{proof}
	The proof is an easy consequence of regularization of $b$ as in \eqref{bregs}, and the passage to the limit follows from the strong convergence of $u^\eps$ to $u$.
\end{proof}

\subsection{The forward ODE flow} \label{ss:forwardflow}

We finally return to the study of the flow \eqref{monotone:flow}, in particular for the forward direction. A candidate for the object $\phi_{t,s}(x)$, $t > s$, a.e. $x$ was already identified in Proposition \ref{P:invertbackward} as the right inverse of the backward flow---note that the full measure set of $x \in \RR^d$ depends on $s$ and $t$. We now connect this right-inverse with the transport equation \eqref{E:MTE}, and exploit the renormalization property to identify $\phi_{t,s}(x)$ as a regular Lagrangian flow, that is, for a.e. $x \in \RR^d$, an absolutely continuous solution of the integral equation for \eqref{monotone:flow} with control on the compressibility.

\subsubsection{Properties of the right inverse}

We first record more properties of the right-inverse of the backward flow identified in Proposition \ref{P:invertbackward}. From now on, for $0 \le s \le t \le T$, we always denote by $\phi_{t,s}$ the version of the right-inverse of $\phi_{s,t}$ which is continuous almost everywhere (such a version is guaranteed to exist by Proposition \ref{P:invertbackward}). 

\begin{theorem}\label{T:regularLagrange}
	For any $t \in (0,T]$, $(s,x) \mapsto \phi_{s,t}(x)$ is (coordinate-by-coordinate) the duality solution of \eqref{E:MTE} with terminal value $x$ at time $t$. For all $1 \le p < \oo$,
	\[
		\phi_{\cdot, s} \in C([s,T], L^p_\loc(\RR^d)) \quad \text{and} \quad \phi_{t,\cdot} \in C([0,t] , L^p_\loc(\RR^d)),
	\]
	and there exists a constant $C > 0$ such that, for all $0 \le s \le t \le T$ and $x \in \RR^d$,
	\[
		|\phi_{t,s}(x)| \le C(1 + |x|) \quad \text{and} \quad
		\norm{\phi_{t,s}}_{BV_\loc} \le C.
	\]
	Finally, if $(b^\eps)_{\eps > 0}$ is as in \eqref{bregs} and $\phi^\eps_{t,s}$ is the corresponding forward flow, then, for all $1 \le p < \oo$,
	\[
		\lim_{\eps \to 0} \phi^\eps_{\cdot,s} = \phi_{\cdot,s} \quad \text{strongly in } C([s,T], L^p_\loc(\RR^d)) \quad \text{and} \quad
		\lim_{\eps \to 0} \phi^\eps_{t,\cdot} = \phi_{t,\cdot} \quad \text{strongly in } C([0,t], L^p_\loc(\RR^d)),
	\]
	and the convergence also holds in the weak-$\star$ sense in $L^\oo_\loc$.
\end{theorem}

\begin{proof}
	For $\eps > 0$ and $(b^\eps)_{\eps > 0}$ as in \eqref{bregs}, it is standard that, for $t \in (0,T]$, the vector-valued solution of 
	\[
		\frac{\del u^\eps}{\del s} + b^\eps \cdot \nabla u^\eps = 0 \quad \text{in } (0,t) \times \RR^d, \quad u^\eps(t,x) = x
	\]
	is given by $u^\eps(s,x) = \phi^\eps_{t,s}(x)$ for $s \in [0,t]$, where $\phi^\eps$ is the flow corresponding to $b^\eps$. By Theorem \ref{T:MTE}, we have the given convergence statements, as $\eps \to 0$, of $\phi^\eps$ to the vector valued duality solution $u$ of \eqref{E:MTE} in $[0,t] \times \RR^d$ with terminal value $u(t,\cdot) = x$.
	
	The flow property for smooth $b^\eps$ yields, for $0 \le s \le t \le T$ and $x \in \RR^d$, $\phi^\eps_{s,t}(\phi^\eps_{t,s}(x))$. By Lemma \ref{L:backwardsflow} and the above strong $L^p$-convergence statement, we may take $\eps \to 0$ to obtain $\phi_{s,t}(u(s,x)) = x$, and then, by Proposition \ref{P:invertbackward}, we must have $u(s,x) = \phi_{t,s}(x)$. The other statements now follow immediately in view of Theorem \ref{T:MTE}. Note that we are using that, for $s \in [0,T)$, the map $[s,T] \times \RR^d \ni (t,x) \mapsto \phi_{t,s}(x)$ is the duality solution of the initial value problem
	\[
		\frac{\del \tilde u}{\del t} - b(t,x) \cdot \nabla \tilde u = 0 \quad \text{in } [s,T] \times \RR^d, \quad \tilde u(s,x) = x,
	\]
	whose theory can be treated exactly as for \eqref{E:MTE}.
	
\end{proof}	
	
\subsubsection{The regular Lagrange property}

We now observe that there is a representation formula for the duality solution of the transport equation \eqref{E:MTE}.
	
\begin{theorem}\label{T:MTErep}
	Let $1 \le p \le \oo$. Then there exists a constant $C > 0$ depending only on $p$ and the constant in \eqref{A:monotoneB} such that, for all $F \in L^p_\loc \cap C$, $R > 0$, and $0 \le s \le t \le T$, 
	\begin{equation}\label{RL:Lp}
		\norm{F \circ \phi_{t,s}}_{L^p(B_R)} \le C \norm{F}_{L^p(B_{R+C})}.
	\end{equation}
	In particular, for any $A \subset \RR^d$ with finite Lebesgue measure,
	\begin{equation}\label{RL:sets}
		\left| \left\{ x : \phi_{t,s}(x) \in A \right\} \right| \le C |A|.
	\end{equation}
	If $u_T \in L^p_\loc(\RR^d)$, then the duality solution of \eqref{E:MTE} is given by
	\begin{equation}\label{formula:ATE}
		u(t,x) = u_T(\phi_{T,t}(x)).
	\end{equation}
	If $u_T$ has a version which is continuous almost everywhere, then, for $t < T$, $u(t,\cdot)$ also has a version that is continuous almost everywhere.
\end{theorem}

\begin{remark}
	When $u_T$ is not continuous, then \eqref{formula:ATE} must be interpreted as the continuous extension of the operator $u_T \mapsto u_T \circ \phi_{T,t}$ to $u_T \in L^p_\loc$, which is well-defined in view of the estimate \eqref{RL:Lp}.
\end{remark}

\begin{remark}\label{R:regLagrange}
	The estimate \eqref{RL:sets} is called the regular Lagrange property. It reinforces the fact that $\phi_{t,s}$ does not concentrate in sets of measure zero.
\end{remark}

\begin{remark}
	The propagation of almost-everywhere continuity is a consequence of the same property for the forward flow (Proposition \ref{P:invertbackward}). Note that it is not true in general that a function $u \in BV_\loc(\RR^d)$ is continuous almost everywhere, unless $d = 1$.
\end{remark}

\begin{proof}[Proof of Theorem \ref{T:MTErep}]
	For continuous $u_T$, the representation formula is an immediate consequence of the renormalization property Theorem \ref{T:MTErenorm} and Theorem \ref{T:regularLagrange}. The estimate \eqref{RL:Lp} then follows from Theorem \ref{T:MTE}, and \eqref{RL:sets} is obtained by taking $p = 1$ and $F = \ind_A$.
	
	For the claim about almost everywhere continuity, define
	\[
		A := \left\{ y \in \RR^d: u_T \text{ is not continuous at }y\right\}.
	\]
	Then $|A| = 0$, and then \eqref{RL:sets} gives, for $0 \le t < T$,
	\[
		\left| \left\{ x \in \RR^d: u_T \text{ is not continuous at } \phi_{t,T}(x) \right\} \right| = 0.
	\]
	It follows that $u_T$ is continuous at $\phi_{t,T}(x)$ for a.e. $x$. By Proposition \ref{P:invertbackward}, $\phi_{t,T}$ is continuous almost everywhere, and the result follows.
\end{proof}

Recalling the duality relationship between \eqref{E:MCE} and \eqref{E:MTE} from Theorem \ref{T:MTE}, we then have the following.

\begin{corollary}\label{C:pushforward}
	For any $1 \le p \le \oo$ and $f_0 \in L^p_\loc(\RR^d)$, the good solution $f$ of \eqref{E:MCE} is given at time $t > 0$ by $\phi_{t,0}^\# f_0$.
\end{corollary}

\begin{remark}
	The regular Lagrange property says that the measure $\phi_{t,0}^\# f_0$ is well-defined and absolutely continuous with respect to Lebesgue measure, with a density in $L^p_\loc$. If $f_0$ is the density for a probability measure, that is, $f_0 \in L^1_+(\RR^d)$ and $\int f_0 = 1$, then $f(t,\cdot)$ is the law at time $t$ of the stochastic process $\phi_{t,0}(X)$, where $X$ is a random variable with density $f_0$.
\end{remark}

A consequence of renormalization and the regular Lagrange property is the fact that the forward flow $\phi_{t,s}$ solves the ODE \eqref{monotone:flow} for a.e. initial $x \in \RR^d$. A first step is the following lemma.

\begin{lemma}\label{L:RHS}
	For all $p \in [1,\oo)$ and $s \in [0,T)$, $\{(t,x) \mapsto b(t,\phi_{t,s}(x))\} \in L^1([0,T], L^p_\loc(\RR^d))$. If $(b^\eps)_{\eps > 0}$ is as in \eqref{bregs} and $(\phi^\eps)_{\eps > 0}$ is the corresponding flow, then, for all $R > 0$,
	\[
		\lim_{\eps \to 0}\int_s^T \norm{ b^\eps(t, \phi^\eps_{t,s}) - b(t,\phi_{t,s}) }_{L^p(B_R)} dt= 0.
	\]
\end{lemma}

\begin{proof}
	The first claim follows from \eqref{RL:Lp}: there exists $C > 0$ independent of $s$ and $R$ such that, for all $t \in [0,T]$, $\norm{b(t,\phi_{t,s})}_{L^p(B_R)} \le C \norm{b(t,\cdot)}_{L^p(B_{R+C})}$.

	For $\delta > 0$ and $0 \le s \le t \le T$, we write
	\begin{align*}
		 \norm{b^\eps(t, \phi^\eps_{t,s}) - b(t,\phi_{t,s})}_{L^p(B_R)}
		 &\le
			 \norm{b^\eps(t, \phi^\eps_{t,s}) - b^\delta(t,\phi^\eps_{t,s}}_{L^p(B_R)}
			 + 
			 \norm{b^\delta(t, \phi^\eps_{t,s}) - b^\delta(t,\phi_{t,s}}_{L^p(B_R)}\\
			&+
			 \norm{b^\delta(t, \phi_{t,s}) - b(t,\phi_{t,s}}_{L^p(B_R)}.
	\end{align*}
	By \eqref{RL:Lp}, for some $C > 0$ independent of $\delta$, $\eps$, $s$, and $t$,
	\[
		\norm{b^\eps(t, \phi^\eps_{t,s}) - b^\delta(t,\phi^\eps_{t,s})}_{L^p(B_R)}
		\le C \norm{b^\eps(t,\cdot) - b^\delta(t,\cdot)}_{L^p(B_{R+C})}
	\]
	and
	\[
		\norm{b^\delta(t, \phi_{t,s}) - b(t,\phi_{t,s})}_{L^p(B_R)}
		\le C \norm{b^\delta(t,\cdot) - b(t,\cdot)}_{L^p(B_{R+C})}.
	\]
	The smoothness of $b^\delta$ implies that, for all $t \in [s,T]$, as $\eps \to 0$, $b^\delta(t,\phi^\eps_{t,s})$ converges a.e. to $b^\delta(t,\phi_{t,s})$. Sending $\eps \to 0$ and using dominated convergence, we thus have
		\[
			\limsup_{\eps \to 0} \int_s^T
			 \norm{b^\eps(t, \phi^\eps_{t,s}) - b(t,\phi_{t,s})}_{L^p(B_R)}dt
			 \le C \int_s^T \norm{b^\delta(t,\cdot) - b(t,\cdot)}_{L^p(B_{R+C})}dt.
		\]
		The proof of the claim is finished upon sending $\delta \to 0$ and again using dominated convergence.
\end{proof}

\begin{theorem}\label{T:forwardflow}
	Fix $1 \le p < \oo$ and $s \in [0,T)$. Then
	\[
		\left\{ (t,x) \mapsto \phi_{t,s}(x) \right\} \in L^p_\loc(\RR^d, W^{1,1}([0,T]) ),
	\]
	and, for a.e. $x \in \RR^d$, $[s,T] \ni t \mapsto \phi_{t,s}(x)$ is an absolutely continuous solution of
	\[
		\phi_{t,s}(x) = x + \int_s^t b(r, \phi_{r,s}(x))dr.
	\]
	If $(b^\eps)_{\eps > 0}$ satisfy \eqref{bregs} and $\phi^\eps$ is the corresponding flow, then, for all $R > 0$,
	\[
		\lim_{\eps \to 0} \norm{ \phi^\eps_{\cdot,s} - \phi_{\cdot,s} }_{L^p(B_R, W^{1,1}([s,T])} = 0.
	\]
	For all $0 \le r \le s \le t \le T$, $\phi_{t,r} = \phi_{t,s} \circ \phi_{s,r}$ a.e.
\end{theorem}
	
\begin{remark}
	The fact that $\del_t \phi_{t,\cdot} \in L^1$ is due to the fact that we are assuming the weakest possible integrability of $b$ in the time variable. If $b \in L^q$ for some $q >1$, then the forward flow belongs to $W^{1,p}$ for any $p \le q$.
\end{remark}

\begin{remark}
	The composition $\phi_{t,s} \circ \phi_{s,r}$ is made sense of due to \eqref{RL:Lp} and the fact that the forward flow takes values in $L^p_\loc(\RR^d)$.
\end{remark}
	
\begin{proof}[Proof of Theorem \ref{T:forwardflow}]
	For $\eps > 0$, we have $\del_t \phi^\eps_{t,s}(x) = b^\eps(t, \phi^\eps_{t,s}(x))$. By Lemma \ref{L:RHS}, sending $\eps \to 0$, we see that the distribution $\del_t \phi_{t,s}(x)$ satisfies, in the distributional sense, $\del_t \phi_{t,s}(x) = b(t,\phi_{t,s}(x))$, and therefore, for all $R > 0$,
	\[
		\norm{ \int_s^T |\phi_{t,s}|dt }_{L^p(B_R)} \le \int_s^T \norm{\phi_{t,s}}_{L^p(B_R)}dt < \oo.
	\]
	The convergence claim and the solvability of the ODE follow immediately in view of the fact that $\phi^\eps_{s,s}(x) = \phi_{s,s}(x) = x$ for all $\eps > 0$ and $x \in \RR^d$.
	
	To prove the last claim, we note that the equality $\phi_{r,t} \circ \phi_{t,r} = \Id$ holds as functions in $L^p_\loc$, and, in view of the flow property of the backward flow,
	\[
		\phi_{r,t} \circ ( \phi_{t,s} \circ \phi_{s,r})
		= \phi_{r,s} \circ \phi_{s,t} \circ \phi_{t,s} \circ \phi_{s,r} = \phi_{r,s} \circ \phi_{s,r} = \Id.
	\]
	It follows from Proposition \ref{P:invertbackward} that $\phi_{t,r} = \phi_{t,s} \circ \phi_{s,r}$ a.e., as desired.
\end{proof}

We recall that Proposition \ref{P:invertbackward} implies that any right-inverse of the backward flow is determined uniquely almost everywhere. We remark here that this property actually follows from the duality between the transport and continuity equations.
	
\begin{theorem}\label{T:charLagrange}
	Assume $\psi \in C([0,t], L^p_\loc(\RR^d))$ satisfies $\phi_{s,t}(\psi_s(x)) = x$ for all $s \in [0,t]$, for a.e. $x \in \RR^d$. Then $\psi = \phi_{t,\cdot}$.
\end{theorem}

\begin{proof}
	It suffices to show that $u(t,x) = \psi_s(x)$ is the unique (vector-valued) duality solution of \eqref{E:MTE} with terminal data equal to $x$ at time $t$. 
	
	Fix $g \in C_c(\RR^d)$. For a.e. $x \in \RR^d$, if $y = \psi_t(x)$, we have $\phi_{s,t}(y) = x$ by assumption. Therefore, the change of variables formula yields
	\[
		\int_{\RR^d} g(x) \psi_t(x)dx = \int_{\RR^d} g(\phi_{s,t}(y)) y J_{s,t}(y)dy = \int_{\RR^d} f(t,y) y dy,
	\]
	where $f$ is the good solution of the forward continuity equation with initial condition $g$ at time $s$. 
\end{proof}

\begin{remark}
	A corresponding result characterizing $\phi_{\cdot, s}$ on $[s,T]$ follows in exactly the same way, by considering the duality between the IVP and TVP for, respectively, an appropriate transport and continuity equation.
\end{remark}

\begin{remark}\label{R:uniqueness}
	The uniqueness result above demonstrates that the right-inverse property is a crucial property of the forward flow. In other words, it implies that $\phi_{t,s}$ solves the ODE, that it solves the transport PDE in the duality sense, and that it has the regularity properties laid out in Theorems \ref{T:MTErep} and \ref{T:forwardflow}.
\end{remark}

\subsection{Characterizations} 

We now present alternative ways to characterize the solutions of the forward continuity and backward transport equations identified above. Although the PDE \eqref{E:MTE} does not make sense as a distribution, we nevertheless can characterize solutions in a PDE sense through the use of $\sup$- and $\inf$-convolutions. The propagation of almost-everywhere continuity proved in Theorem \ref{T:MTErep} is a crucial ingredient.

By using this characterization in duality with the conservative equation, we then show that nonnegative distributional solutions of \eqref{E:MCE} are unique, and therefore equal to the ``good'' solution identified by the formula \eqref{formula:M}. As a consequence, we finally conclude with the uniqueness of regular Lagrangian flows, forward in time, of the ODE \eqref{monotone:flow}.

\subsubsection{The nonconservative equation: $\sup$ and $\inf$ convolutions}

Given $\delta > 0$ and $u \in L^\oo( \RR^d)$, we define the $\sup$- and $\inf$-convolutions
\[
	u^\delta(x) := \esssup_{y \in \RR^d} \left\{ u(y) - \frac{1}{2\delta}|x-y|^2 \right\}	
\]
and
\[
	u_\delta(x) := \essinf_{y \in \RR^d} \left\{ u(y) + \frac{1}{2\delta}|x-y|^2 \right\}.
\]
These regularizations are common in the theory of viscosity solutions, or generally for equations satisfying a maximum principle in spaces of continuous functions. The supremum and infimum must be essential, because $u$ is only defined almost everywhere.

\begin{lemma}\label{L:supinf}
	Assume that $u \in L^\oo(\RR^d)$ is continuous almost everywhere. Then, for all $\delta > 0$, $u_\delta, u^\delta$ are globally Lipschitz with constant
	\[
		(\esssup u - \essinf u)^{1/2} \delta^{-1/2},
	\]
	and
	\[
		u_\delta \le u \le u^\delta \quad \text{a.e.}
	\]
	As $\delta \to 0$, $u^\delta$ decreases to $u$ and $u_\delta$ increases to $u$ a.e. Finally, the $\esssup$ and $\essinf$ in the definitions of $u^\delta$ and $u_\delta$ can be restricted to respectively $y \in B_{R^\delta(x)}(x)$ and $B_{R_\delta(x)}(x)$, where
	\[
		R^\delta(x) = 2(u^{2\delta}(x) - u^\delta(x))^{1/2} \delta^{1/2}.
	\]
	and
	\[
		R_\delta(x) = 2(u_\delta(x) - u_{2\delta}(x))^{1/2} \delta^{1/2}.
	\]
\end{lemma}

\begin{proof}
	Fix $x \in \RR^d$ and $r > 0$. We thus have
	\[
		u^\delta(x) \ge \esssup_{y \in B_r(x)} u(y) - \frac{r^2}{2\delta}.
	\]
	Sending $r \to 0$, we see that $u^\delta(x) \ge u(t,x)$ whenever $u$ is continuous at $x$, and therefore $u^\delta \ge u$ a.e. Similarly, $u_\delta \le u$ a.e.
	
	We now observe that, if $R > (\esssup u - \essinf u)^{1/2}$, then, for a.e. $y \notin B_{R\delta^{1/2}}$,
	\[
		u(y) - \frac{|x-y|^2}{2\delta} \le \esssup u - R^2 < \essinf u \le u^\delta(x).
	\]
	By also using a similar argument for $u_\delta$, we see that
	\[
		u^\delta(x) := \esssup_{|y-x| \le R\delta^{1/2}} \left\{ u(y) - \frac{1}{2\delta}|x-y|^2 \right\}	
	\]
	and
	\[
		u_\delta(x) := \essinf_{|y-x| \le R\delta^{1/2}} \left\{ u(y) + \frac{1}{2\delta}|x-y|^2 \right\}.
	\]
	It is then straightforward to see that $u^\delta$ and $u_\delta$ are respectively decreasing and increasing pointwise as $\delta$ decreases to $0$, and converge whenever $u$ is continuous at $x$ (and thus a.e.) to $u(x)$.
	
	For fixed $x \in \RR^d$, $\delta > 0$, and $\eta > 0$, define
	\[
		A_{\delta,\eta}(x) := \left\{ y \in \RR^d: u(y) - \frac{|x-y|^2}{2\delta} > u^\delta(x) - \eta \right\}.
	\]
	Then, by definition, $A_{\delta,\eta}(x)$ has nonzero Lebesgue measure. Therefore, for any $x' \in \RR^d$, there exists $y \in A_{\delta,\eta}(x)$ such that
	\[
		u(y) - \frac{|x'-y|^2}{2\delta} \le u^\delta(x'),
	\]
	and so
	\[
		u^\delta(x) - u^\delta(x') \le \frac{|x'-y|^2}{2\delta} - \frac{|x-y|^2}{2\delta} + \eta \le \frac{R}{\delta^{1/2}}|x'-x| + \frac{|x'-x|^2}{\delta} + \eta.
	\]
	Switching the roles of $x$ and $x'$ and using the fact that $\eta$ was arbitrary, we see that, for all $x \in \RR^d$,
	\[
		\limsup_{x' \to x} \frac{|u^\delta(x') - u^\delta(x)|}{|x'-x|} \le \frac{R}{\delta^{1/2}}.
	\]
	We may then let $R$ decrease down to $(\esssup u - \essinf u)^{1/2}$, and the same proof for $u_\delta$ holds.
	
	For any $\eta > 0$ and a.e. $y \in A^\eta_\delta$,
	\[
		u^{2\delta}(x) \ge u(y) - \frac{|x-y|^2}{4\delta} > u^\delta(x) + \frac{|x-y|^2}{2\delta} - \eta,
	\]
	and so
	\[
		|y-x| \le 2( u^{2\delta}(x) - u^\delta(x) + \eta )^{1/2} \delta^{1/2}.
	\]
	Therefore, for a.e. $y$ such that $|y-x| > R^\delta(x)$, we must have $u(y) - \frac{|x-y|^2}{2\delta} < u^\delta(x)$, and the statement about restricting the $\esssup$ follows. The corresponding result for $u_\delta$ is proved in the same way.
\end{proof}

\begin{theorem}\label{T:MTEchar}
	Assume $u \in C([0,T], L^1_\loc(\RR^d)) \cap L^\oo([0,T] \times \RR^d)$ is continuous almost everywhere and $u(T,\cdot) = u_T \in L^\oo(\RR^d)$. Then $u$ is the duality solution of \eqref{E:MTE} if and only if there exist $r^\delta, r_\delta \in L^1_\loc([0,T] \times \RR^d))$ such that $\lim_{\delta \to 0} r^\delta = \lim_{\delta \to 0} r_\delta = 0$ in $L^1_\loc$, and the $\sup$- and $\inf$-convolutions
	\[
		u^\delta(t,x) := \esssup_{y \in \RR^d} \left\{ u(t,y) - \frac{1}{2\delta}|x-y|^2 \right\}	
	\]
	and
	\[
		u_\delta(t,x) := \essinf_{y \in \RR^d} \left\{ u(t,y) + \frac{1}{2\delta}|x-y|^2 \right\}
	\]
	satisfy in the sense of distributions on $[0,T] \times \RR^d$ the inequalities
	\[
		\frac{\del u^\delta}{\del t} + b(t,x) \cdot \nabla u^\delta \le r^\delta(t,x)
		\quad \text{and} \quad
		\frac{\del u_\delta}{\del t} + b(t,x) \cdot \nabla u_\delta \ge -r_\delta(t,x).
	\]
\end{theorem}

\begin{proof}
	Assume first that the $\sup$- and $\inf$-convolutions have the stated properties. For standard mollifiers $(\rho_\eta)_{\eta > 0}$ on $\RR$, define $u^\delta_\eta(t,x) = (u^\delta(\cdot, x) *_t \rho_\eta)(t)$ and $u_{\delta,\eta}(t,x) = (u_\delta(\cdot, x) *_t \rho_\eta)(t)$. Then, by Lemma \ref{L:supinf}, $u^\delta_\eta$ and $u_{\delta,\eta}$ are Lipschitz continuous on $[0,T] \times \RR^d$, and satisfy a.e. in $[0,T] \times \RR^d$
	\[
		\frac{\del u^\delta_\eta}{\del t} + b(t,x) \cdot \nabla u^\delta_\eta \le r^\delta_\eta(t,x)
		\quad \text{and} \quad
		\frac{\del u_{\delta,\eta}}{\del t} + b(t,x) \cdot \nabla u_{\delta,\eta} \ge -r_{\delta,\eta}(t,x),
	\]
	where
	\[
		r^\delta_\eta(t,x) = (r^\delta(\cdot,x) *_t \rho_\eta)(t) + \int_\RR (b(t,x) - b(s,x)) \cdot \nabla u^\delta(s,x) \rho_\eta(s-t)ds
	\]
	and
	\[
		r_{\delta,\eta}(t,x) = (r_\delta(\cdot,x) *_t \rho_\eta)(t) + \int_\RR (b(t,x) - b(s,x)) \cdot \nabla u_\delta(s,x) \rho_\eta(s-t)ds.
	\]
	The (local) boundedness of $b$, $\nabla u_\delta$, and $\nabla u^\delta$ then allow us to invoke the dominated convergence theorem to say that, for fixed $\delta$, $\lim_{\eta \to 0} r^\delta_\eta = r^\delta$ and $\lim_{\eta \to 0} r_{\delta,\eta} = r_\delta$ in $L^1_\loc$.
	
	Now let $f_0 \in C_c(\RR^d)$ be nonnegative and let $f$ be the ``good'' solution of \eqref{E:MCE}. In view of the nonnegativity of $J$, $f$ given by \eqref{formula:M} is nonnegative on $[0,T] \times \RR^d$, and the bounds for the backward flow in Lemma \ref{L:backwardsflow} imply that $f$ has compact support in $[0,T] \times \RR^d$. By Theorem \ref{T:MCE}, $f$ is a distributional solution, and therefore
	\begin{align*}
		\int_{\RR^d} f(T,x) u^\delta_\eta(T,x)dx - \int_{\RR^d} f_0(x) u^\delta_\eta(0,x)dx
		&= \int_0^T \int_{\RR^d}f(t,x) \left[ \del_t u^{\delta,\eta}(t,x) + b(t,x) \cdot \nabla u^{\delta,\eta}(t,x) \right]dxdt\\
		&\le \int_0^T \int_{\RR^d} f(t,x) r^\delta_\eta(t,x)dxdt.
	\end{align*}
	Sending first $\eta \to 0$ and then $\delta \to 0$, using Lemma \ref{L:supinf} and the dominated convergence theorem, we conclude that
	\[
		\int_{\RR^d} f(T,x) u_T(x)dx \le \int_{\RR^d} f_0(x) u(0,x)dx.
	\]
	Arguing similarly with $u_{\delta,\eta}$ as a test function, we achieve the opposite inequality. By linearity, the duality identity holds for any $f_0 \in L^\oo$ with bounded support, and we conclude that $u$ is the unique duality solution.
	
	Assume now conversely that $u$ is the duality solution. Let $(b^\eps)_{\eps > 0}$ be as in \eqref{bregs}, let $u^\eps$ be the corresponding solution, and define
	\[
		u^{\eps,\delta}(t,x) := \sup_{y \in \RR^d} \left\{ u^\eps(t,y) - \frac{1}{2\delta}|x-y|^2 \right\}	
	\]
	and
	\[
		u^\eps_\delta(t,x) := \inf_{y \in \RR^d} \left\{ u^\eps(t,y) + \frac{1}{2\delta}|x-y|^2 \right\}.
	\]
	By Lemma \ref{L:supinf}, for fixed $\delta > 0$, $u^{\eps,\delta}$ and $u^\eps_\delta$ are Lipschitz continuous in the space variable, uniformly over $[0,T] \times \RR^d$ and $\eps > 0$. Moreover, the $\sup$ and $\inf$ are actually a $\max$ and $\min$, and may be restricted to
	\[
		|y-x| \le (\max u_0 - \min u_0)^{1/2}\delta^{1/2}
	\]
	(note that we have used the maximum principle for the transport equation to control the maximum and minimum of $u^\eps$ and $u_\eps$). We may alternatively restrict the $y$ for which the maximum in the definition of $u^{\eps,\delta}(t,x)$ is attained to satisfy
	\begin{equation}\label{udeltaepsmax}
		|y-x| \le 2(u^{\eps,2\delta}(t,x) - u^{\eps,\delta}(t,x))^{1/2} \delta^{1/2},
	\end{equation}
	and the minimum in the definition of $u^\eps_\delta$ is attained by $y$ satisfying
	\begin{equation}\label{udeltaepsmin}
		|y-x| \le 2 (u^\eps_\delta(t,x) - u^\eps_{2\delta}(t,x))^{1/2} \delta^{1/2}.
	\end{equation}
	Standard properties of envelopes then give the identities, for any $(t,x) \in [0,T] \times \RR^d$,
	\[
		\frac{\del u^{\eps,\delta}}{\del t}(t,x) = \frac{\del u^\eps}{\del t}(t,y)
		\quad \text{and} \quad
		\nabla u^{\eps,\delta}(t,x) = \nabla u^\eps(t,y) = \frac{y-x}{\delta}
	\]
	for some $y$ satisfying \eqref{udeltaepsmax}. Therefore
	\[
		\del_t u^{\eps,\delta}(t,x) = - b^\eps(t,y) \cdot \nabla u^{\eps,\delta}(t,x), 
	\]
	from which we deduce that $u^{\eps,\delta}$ is uniformly Lipschitz continuous in the time variable over $[0,T] \times B_R$ for any $R > 0$, independently of $\eps$. Further developing the equality gives
	\begin{equation}\label{E:u^eps,delta}
	\begin{split}
		\frac{\del u^{\eps,\delta}}{\del t}(t,x) + b^\eps(t,x) \cdot \nabla u^{\eps,\delta}(t,x)
		&= \frac{\del u^{\eps}}{\del t}(t,y) + b^\eps(t,x) \cdot \nabla u^{\eps}(t,y)\\
		&= -(b^\eps(t,x) - b^\eps(t,y)) \cdot \frac{x-y}{\delta} \\
		&\le C_0(t) \frac{|x-y|^2}{\delta} \le 4C_0(t)(u^{\eps,2\delta}(t,x) - u^{\eps,\delta}(t,x)).
	\end{split}
	\end{equation}
	We similarly have that $u^\eps_\delta$ is Lipschitz continuous in the time variable, locally in space, uniformly over $\eps > 0$, and
	\begin{equation}\label{E:u^eps_delta}
		\frac{\del u^{\eps}_\delta}{\del t}(t,x) + b^\eps(t,x) \cdot \nabla u^{\eps}_\delta(t,x) \ge
		-4C_0(t) ( u^\eps_\delta(t,x) - u^\eps_{2\delta}(t,x)).
	\end{equation}
	
	We now claim that, as $\eps \to 0$, $u^{\eps,\delta}$ and $u^\eps_\delta$ converge pointwise to respectively $u^\delta$ and $u_\delta$, and then, by the uniform-in-$\eps$ Lipschitz regularity, the convergence is locally uniform. To see this, fix $x \in \RR^d$ and $\eta > 0$, and let $A \subset \RR^d$ be a set of positive measure such that
	\[
		u^\delta(t,x) \le u(t,y) - \frac{|x-y|^2}{2\delta} + \eta.
	\]
	We then have, for all $y \in A$,
	\[
		u^{\delta,\eps}(t,x) \ge u^\eps(t,y) - \frac{|x-y|^2}{2\delta}.
	\]
	For at least one such $y$, we then have $u^\eps(t,y) \xrightarrow{\eps \to 0} u(t,y)$, and we thus have
	\[
		\limsup_{\eps \to 0} \left( u^\delta(t,x) - u^{\delta,\eps}(t,x) \right) \le \eta.
	\]
	It follows that $\limsup_{\eps \to 0} \left( u^\delta(t,x) - u^{\delta,\eps}(t,x) \right) \le 0$ since $\eta$ was arbitrary. 
	
	Now, there exists a full measure set $B \subset \RR^d$ such that, for all $y \in B$,
	\[
		u^\delta(t,x) \ge u(t,y) - \frac{|x-y|^2}{2\delta} \quad \text{and} \quad \lim_{\eps \to 0} u^\eps(t,y) = u(t,y).
	\]
	In view of the continuity of $u^\eps(t,\cdot)$, there exists a bounded (independently $\eps$) sequence $(y_n)_{n \in \NN} \subset B$ such that
	\[
		\rho_n := u^{\delta,\eps}(t,x) -  \left\{ u^\eps(t,y_n) - \frac{|x-y_n|^2}{2\delta} \right\}
	\]
	satisfies $\lim_{n\to\oo} \rho_n = 0$. Therefore, for all $n$,
	\[
		u^{\delta,\eps}(t,x) - u^\delta(t,x) \le u^\eps(t,y_n) - u(t,y_n) + \rho_n.
	\]
	Sending $\eps \to 0$ gives $\limsup_{\eps \to 0} (u^{\delta,\eps}(t,x) - u^\delta(t,x)) \le \rho_n$, and the proof of pointwise convergence is finished upon sending $n \to \oo$. The exact same argument can be used for the pointwise convergence of $u^\eps_\delta$ to $u_\delta$.
	
	It then follows that, for fixed $\delta$, as $\eps \to 0$, $\nabla u^{\eps,\delta}$ and $\nabla u^{\eps}_\delta$ converge weak-$\star$ in $L^\oo$ to $\nabla u^\delta$ and $\nabla u_\delta$ respectively, while $b^\eps$ converges in $L^1_\loc$ to $b$. We may then take $\eps \to 0$ in \eqref{E:u^eps,delta} and \eqref{E:u^eps_delta} to obtain the distributional inequalities
	\[
		\frac{\del u^{\delta}}{\del t}(t,x) + b(t,x) \cdot \nabla u^{\delta}(t,x) 
		\le 4C_0(t)(u^{2\delta}(t,x) - u^{\delta}(t,x)) =: r^\delta(t,x)
	\]
	and
	\[
		\frac{\del u_\delta}{\del t}(t,x) + b(t,x) \cdot \nabla u_\delta(t,x) \ge
		-4C_0(t) ( u_\delta(t,x) - u_{2\delta}(t,x)) =: -r_\delta(t,x).
	\]
	By Lemma \ref{L:supinf} and the almost-everywhere continuity of $u$, the right-hand sides of both inequalities converge a.e. to $0$ as $\delta \to 0$, and, by the uniform boundedness in $\delta$ of $u^\delta$ and $u_\delta$ and the dominated convergence theorem, $r^\delta$ and $r_\delta$ both converge in $L^1_\loc$ to $0$ as $\delta \to 0$.
\end{proof}

\subsubsection{The conservative equation: uniqueness of nonnegative solutions}

We observe that, in the first implication in the proof of Theorem \ref{T:MTEchar}, it was proved that $u$ was a duality solution by proving the duality identity relative to the a ``good'' nonnegative solution. However, it was only explicitly used that $f$ was a distributional solution. Therefore, after having proved the equivalence in Theorem \ref{T:MTEchar}, we arrive at the following.

\begin{theorem}\label{T:MCEunique}
	Suppose that $f \in C([0,T], L^p_\loc(\RR^d))$ is a distributional solution of \eqref{E:MCE} and $f \ge 0$. Then $f(t,x) = f(0,\phi_{0,t}(x))J_{0,t}(x)$.
\end{theorem}

\begin{proof}
	Fix $t > 0$ and $v \in C_c(\RR^d)$, and let $u \in C([0,t], L^1_\loc(\RR^d)) \cap L^\oo([0,t] \times \RR^d)$ be the duality solution of \eqref{E:MTE} with terminal data $v$ at time $t$. Then, by Theorem \ref{T:MTErep}, $u$ is continuous almost everywhere in $[0,t] \times \RR^d$. Arguing exactly as in the first part of Theorem \ref{T:MTEchar}, using the nonnegativity of $f$, we arrive at the equality
	\[
		\int_{\RR^d} f(t,x) v(x)dx = \int_{\RR^d} f(0,x)u(0,x)dx.
	\]
	Since $v$ was arbitrary, it follows from the definition of duality solutions that $f(t,x)$ must be given by \eqref{formula:M}.
\end{proof}

We then have the following corollary about characterizing the good solution even when $f$ is signed:

\begin{corollary}
	A function $f \in C([0,T], L^p_\loc(\RR^d))$ is the good solution of \eqref{E:MCE} if and only if $f$ and $|f|$ are both solutions in the sense of distributions.
\end{corollary}

\begin{proof}
	That this property is satisfied by the good solution was already pointed out (Corollary \ref{C:Mrenorm}). Suppose now that $f$ and $|f|$ are both distributional solutions. It follows that $f_+ = \frac{1}{2} (f + |f|)$ and $f_- = \frac{1}{2} (|f| - f)$ are distributional solutions, and, since $f_+ \ge 0$ and $f_- \ge 0$, they are both the good solutions. Therefore $f = f_+ - f_-$ is a good solution by linearity.
\end{proof}

\subsubsection{Uniqueness of regular Lagrangian flows}

We can finally establish the uniqueness for the forward flows of the ODE \eqref{monotone:flow}

\begin{theorem}\label{T:Uforwardflow}
	For every $s \in [0,T]$ and almost every $x \in \RR^d$, $\phi_{t,s}(x)$ is the unique absolutely continuous solution of \eqref{monotone:flow}.
\end{theorem}

\begin{proof}
	This is a consequence of Theorem \ref{T:MCEunique} and the superposition principle of Ambrosio \cite[Theorem 3.1]{ACetraro}.
\end{proof}

\subsection{Some remarks for second order equations} 

We next investigate the second-order analogues of \eqref{E:MCE} and \eqref{E:MTE}. As mentioned earlier, we are not able to treat the most general case in which $\sigma$ is a regular function of $x$. This is due to the fact that Lemma \ref{L:backwardsstochflow} only gives regularity of the backward stochastic flow in $C^{0,1-\eps}$ for $0 < \eps < 1$. As a consequence, defining the Jacobian and using it to analyze the right-inverse of the flow is not possible in general. Our results in this case are limited to stochastic flows for which the coefficient $\sigma$ in front of the Wiener process is constant in the space variable. The generalization to regular but nonconstant $\sigma$ will be the subject of future work.

\subsubsection{The expansive stochastic flow with constant noise coefficient}

The stochastic analogue of the forward flow \eqref{monotone:flow} is
\begin{equation}\label{stochflow:general}
	d_t \Phi_{t,s}(x) = b(t, \Phi_{t,s}(x))dt + \sigma(t, \Phi_{t,s}(x))dW_t, \quad t \in [s,T], \quad \Phi_{s,s}(x) = x,
\end{equation}
where $\sigma: [0,T] \times \RR^d \to \RR^{d \times m}$ is some matrix-valued map. As we shall see, this general setting is out of the reach at the moment, and we thus assume
\begin{equation}\label{A:constantnoise}
	\sigma \in L^2([0,T], \RR^{d \times m})
\end{equation}
is constant in the space variable. We then consider the forward stochastic flow
\begin{equation}\label{monotone:fwdstochflow}
	d\Phi_{t,s}(x) = b(t,\Phi_{t,s}(x))dt + \sigma_t dW_t, \quad t \in [s,T], \quad \Phi_{s,s}(x) = x.
\end{equation}
Formally defining
\[
	\tilde \Phi_{t,s}(x) := \Phi_{t,s}(x) - \underbrace{\int_s^t \sigma_r dW_r}_{:= M_t - M_s}
\]
leads to the random ODE
\begin{equation}\label{randomODE}
	\del_t \tilde \Phi_{t,s}(x) = b\left( t, \tilde \Phi_{t,s}(x) + M_t - M_s \right), \quad t \in [s,T], \quad \tilde \Phi_{s,s}(x) = x.
\end{equation}
We now invoke the theory of the previous subsections to obtain the following:

\begin{theorem}\label{T:fwdstochflowconst}
	For every $s \in [0,T)$, with probability one, there exists a unique $\Phi_{\cdot,s} \in C([s,T], L^p_\loc(\RR^d)) \cap L^p_\loc(\RR^d, C([s,T]))$ such that, for a.e. $x \in \RR^d$, 
	\[
		\Phi_{t,s}(x) = x + \int_s^t b(r, \Phi_{r,s}(x))dr + \int_s^t \sigma_r dW_r.
	\]
	If $(b^\eps)_{\eps > 0}$ are as in \eqref{bregs} and $\Phi^\eps$ is the unique stochastic flow solving \eqref{monotone:fwdstochflow} with drift $b^\eps$, then, with probability one, as $\eps \to 0$, $\Phi^\eps$ converges in $C([s,T], L^p_\loc(\RR^d))$ and in $L^p_\loc(\RR^d, C([s,T]))$ to $\Phi$.
\end{theorem}

\begin{proof}
	This follows upon applying the results of Theorems \ref{T:forwardflow} and \ref{T:Uforwardflow} to the random ODE \eqref{randomODE}.
\end{proof}

\subsubsection{A priori estimates for the second-order nonconservative equation}

We next relate the forward stochastic flow from the previous subsection to the terminal value problem for a certain second-order, nonconservative equation. This will be done with the use of a priori $L^p$ and $BV$ estimates, which lead to useful compactness results, just as for the first order case.

We begin with the more general problem
\begin{equation}\label{E:2MTE}
	-\del_t u - \tr[ a(t,x) \nabla^2 u] + b(t,x) \cdot \nabla u = 0 \quad \text{in } (0,T) \times \RR^d, \quad u(T,\cdot) = u_T,
\end{equation}
where
\begin{equation}\label{A:C11noise}
	a(t,x) = \frac{1}{2} \sigma(t,x) \sigma(t,x)^T, \quad \sigma \in L^2([0,T], C^{1,1}(\RR^d, \RR^{d \times m}) );
\end{equation}	
notice that, although we allow $\sigma$ to be nonconstant here, we require more regularity for $\sigma$ than in Section \ref{sec:antimonotone}.

\begin{lemma}\label{L:b=0BV}
	There exists $C \in L^1_+([0,T])$ depending only the $C^{1,1}$ norm of $\sigma$ such that, if $u$ is a smooth solution of
	\[
		-\del_t u - \tr[ a(t,x) \nabla^2 u] = 0 \quad \text{in }(0,T) \times \RR^d, \quad u(T,\cdot) = u_T,
	\]
	then
	\[
		\norm{u(t,\cdot)}_{BV(\RR^d)} \le \exp\left( \int_t^T C(s)ds \right) \norm{ u_T}_{BV(\RR^d)}.
	\]
\end{lemma}

\begin{proof}
For $(t,x,z) \in [0,T] \times \RR^d \times \RR^d$, set $w(t,x,z) = \nabla u(t,x) \cdot z$. Then $w$ solves the parabolic PDE
\[
	\frac{\del w}{\del t} - \tr[A(t,x,z)\nabla^2_{(x,z)}w] = 0 \quad \text{in }(0,T) \times \RR^{2d},
\]
where
\[
	A(t,x,z) = 
	\frac{1}{2}
	\begin{pmatrix}
		\sigma(t,x) \\
		z \cdot \nabla \sigma(t,x)
	\end{pmatrix}
	\begin{pmatrix}
		\sigma(t,x)^T & z \cdot \nabla \sigma(t,x)^T
	\end{pmatrix}.
\]
After a routine regularization argument, using the convexity of $w \mapsto |w|$, 
\begin{equation}\label{|w|sub}
	\frac{\del |w|}{\del t} - \tr[A(t,x,z) \nabla^2_{(x,z)}|w| ] \le 0 \quad \text{in }(0,T) \times \RR^d \times \RR^d.
\end{equation}
For some $m > d +1$, let $\phi \in C^\oo_+([0,\oo))$ be such that, for some universal $C > 0$,
\begin{equation}\label{nicetest}
	\phi(r) = \frac{1}{r^m} \quad \text{for } r \ge 1 \quad \text{and} \quad r|\phi'(r)| + r^2 |\phi''(r)| \le C \phi(r) \quad \text{for all } r \ge 0.
\end{equation}
We multiply \eqref{|w|sub} by $\phi(|z|)$ and integrate in $(x,z) \in \RR^d \times \RR^d$. Then \eqref{A:C11noise} and \eqref{nicetest} imply that for some $C \in L^1_+([0,T])$,
\[
	-\frac{d}{dt} \iint_{\RR^d \times \RR^d} |w(t,x,z)| \phi(|z|)dxdz \le C(t) \iint_{\RR^d \times \RR^d} |w(t,x,z)| \phi(|z|)dxdz.
\]
The proof is then finished by Gr\"onwall's lemma and the fact that
\[
	\iint_{\RR^d \times \RR^d} |w(t,x,z)| \phi(z)dxdz = c_0 \int_{\RR^d} |\nabla u(t,x)| dx,
\]
where $c_0 := \int_{\RR^d} |\nu \cdot z| \phi(|z|) dz$ is finite and independent of $|\nu| = 1$.
\end{proof}

We have already proved an exponential propagation of the $BV$ bounds when $a = 0$ in Lemma \ref{L:MTEapriori}. It is a classical fact for evolution PDEs that, upon using a splitting scheme, that these estimates can be combined, and we immediately have the following:

\begin{lemma}\label{L:2MTEapriori}
	There exists a constant $C \in L^1_+([0,T])$ depending only on the constants in \eqref{A:monotoneB} and \eqref{A:C11noise} such that, if $u$ is a smooth solution of \eqref{E:2MTE}, then
	\[
		\norm{u(t,\cdot)}_{L^p} \le \exp\left( \int_0^t C(s)ds \right) \norm{u_T}_{L^p} \quad \text{and} \quad
		\norm{u(t,\cdot)}_{BV} \le \exp \left( \int_0^t C(s)ds \right) \norm{u_T}_{BV}.
	\]
\end{lemma}

Just as in the first-order case, it is not possible to define $L^p$-distributional solutions of \eqref{E:2MTE}, and the utility of Lemma \ref{L:2MTEapriori} is that it allows to obtain strongly convergent subsequences in $C([0,T], L^p(\RR^d))$ after regularizing the velocity field $b$.

The main question is whether such limiting solutions are unique. This uniqueness was achieved in the first-order case through duality with the conservative equation, and the solution was further characterized with a formula involving the forward flow. In the second-order case, we are constrained to work with constant noise coefficients:
\begin{equation}\label{E:2MTEconst}
	-\del_t u - \tr[ a(t)\nabla^2 u] + b(t,x) \cdot \nabla u = 0 \quad \text{in }(0,T) \times \RR^d, \quad u(T,\cdot) = u_T,
\end{equation}
where $a = \frac{1}{2} \sigma \sigma^T$ as before.

\begin{theorem}\label{T:2MTEconst}
	For $1 < p < \oo$ and $t \in [0,T]$, the map 
	\[
		C_c(\RR^d) \ni u_T \mapsto \EE[ u_T \circ \Phi_{T,t}]
	\]
	extends to a continuous, linear, order-preserving map on $L^p(\RR^d)$, and the function
	\begin{equation}\label{formula:2MTE}
		u(t,x) := \EE[ u_T(\Phi_{T,t}(x))] \quad (t,x) \in [0,T] \times \RR^d
	\end{equation}
	belongs to $C([0,T], L^p(\RR^d))$, and, if $u_T \in BV(\RR^d)$, then $u \in L^\oo([0,T], BV(\RR^d))$.
	
	 If $(b^\eps)_{\eps > 0}$ is as in \eqref{bregs} and $u^\eps$ is the corresponding solution of \eqref{E:2MTEconst}, then, as $\eps \to 0$, $u^\eps$ converges strongly to $u$ in $C([0,T], L^p(\RR^d))$.
\end{theorem}

\begin{proof}
	Assume that $u_T \in C^2(\RR^d) \cap C_c(\RR^d)$. For $b^\eps$ and $u^\eps$ as in the statement of the theorem, we have the standard representation formula $u^\eps(t,x) = \EE[ u_T(\Phi^\eps_{T,t}(x))]$, where $\Phi^\eps$ corresponds to the flow \eqref{monotone:fwdstochflow} with drift $b^\eps$. By Theorem \ref{T:fwdstochflowconst}, for any $t \in [0,T]$, with probability one, $u_T \circ \Phi^\eps_{T,t} \to u_T \circ \Phi_{T,t}$ a.e. in $\RR^d$. On the other hand, by Lemma \ref{L:2MTEapriori}, $(u^\eps)_{\eps > 0}$ is precompact in $C([0,T], L^p(\RR^d))$, and therefore the full sequence converges to $u$ given by \eqref{formula:2MTE}. The $L^p$-bounds and the extension to $u_T \in L^p(\RR^d)$ now follow from the $L^p$ a priori estimates in Lemma \ref{L:2MTEapriori}.
\end{proof}

\subsubsection{Representation formula for the Fokker-Planck equation}

We turn next to the Fokker-Planck equation
\begin{equation}\label{E:2MCE}
	\del_t f - \nabla^2 \cdot (a(t,x) f) + \div(b(t,x) f) = 0 \quad \text{in } (0,T) \times \RR^d, \quad f(0,\cdot) = f_0,
\end{equation}
where once again $a = \frac{1}{2} \sigma \sigma^T$ with $\sigma$ as in \eqref{A:C11noise}. 

The existence of solutions in $C([0,T], L^p(\RR^d))$ is straightforward; we include the proof for convenience.

\begin{theorem}\label{T:2MCEexist}
	For any $f_0 \in L^p(\RR^d)$, $1 \le p \le \oo$, there exists a distributional solution $f \in C([0,T], L^p_\w(\RR^d))$ if $1 \le p < \oo$, or $f \in L^\oo$ if $p = \oo$. Moreover, there exists $C \in L^1_+([0,T])$ depending only on $p$, $C_0(t)$ from \eqref{A:monotoneB} and the $L^2([0,T], C^{1,1}(\RR^d))$ norm of $a$\footnote{In fact, only an upper bound for $\nabla^2 \cdot a = \del_{ij} a_{ij}$ is needed.} such that
	\[
		\norm{f(t,\cdot)}_{L^p} \le \exp \left( \int_0^t C(s)ds \right) \norm{f}_{L^p}.
	\]
\end{theorem}

\begin{proof}
	We do this with the use of a priori estimates, assuming all the data is smooth. The computations may be made rigorous by regularizing $b$, adding a small ellipticity to $a$, and extracting weakly convergent subsequences. 

We then compute
\[
	\del_t |f|^p -\nabla^2 \cdot( a(t,x) |f|^p) + \div (b(t,x) |f|^p) \le (p-1) \left( \nabla^2 \cdot a(t,x) - \div b(t,x) \right)|f|^p,
\]
and so $\del_t \int |f(t,\cdot)|^p \le C(t) \int |f(t,\cdot)^p$ for some $C$ as in the statement of the Theorem. The result now follows from Gr\"onwall's lemma.
\end{proof}

We now explore the possibility of obtaining a formula for the solution, similar to \eqref{formula:M} for the first order equation \eqref{E:MCE}. To do so, it is convenient to reverse time and consider, for fixed $t \in (0,T]$, the equation satisfied by $g^{(t)}(s,x) := f(t-s,x)$:
\[
	-\del_s g^{(t)} -\nabla^2 \cdot(a(t-s,x)g^{(t)}) + \div(b(t-s,x)g^{(t)}) = 0 \quad \text{in } (0,t) \times \RR^d, \quad g^{(t)}(t,\cdot) = f_0.
\]
For $(s,x,\xi) \in [0,t] \times \RR^d \times \RR$, define $G^{(t)}(s,x,\xi) = g^{(t)}(s,x) \xi$. Then
\begin{equation}\label{E:G}
	\begin{dcases}
	-\del_s G^{(t)} - \tr[ A^{(t)}(s,x,\xi) \nabla^2_{x,\xi}G^{(t)}] - B^{(t)}(s,x) \cdot \nabla G^{(t)} - C^{(t)}(s,x) \xi \del_\xi G^{(t)}= 0 & \text{in } (0,t) \times \RR^{d+1}, \\
	G^{(t)}(t,x,\xi) = f_0(x)\xi,
	\end{dcases}
\end{equation}
where
\begin{equation}\label{newdata}
	\left\{
	\begin{split}
	A^{(t)}(s,x,\xi) &= \frac{1}{2}\Sigma^{(t)}(s,x,\xi) \Sigma^{(t)}(s,x,\xi)^T, \quad \Sigma^{(t)}(s,x,\xi) = 
	\begin{pmatrix}
		\sigma \\
		\xi \div \sigma
	\end{pmatrix},\\
	B^{(t)}(s,x) &= -b  + (\sigma \cdot \nabla)\sigma^T, \quad \text{and}\\
	C^{(t)}(s,x) &= -\div\left( b- \div a \right) \\
	&=  -\div b + \tr[ (\sigma \cdot \nabla)(\nabla \cdot \sigma) ] + \frac{1}{2} |\div \sigma|^2 + \frac{1}{2} \tr[ \nabla \sigma \nabla \sigma^T ];
	\end{split}
	\right.
\end{equation}
for brevity, we have suppressed the arguments for $a$, $\sigma$, and $b$, which are all $(t-s,x)$.

For an $m$-dimensional Wiener process $W$ on $[0,t]$ and a fixed $s \in [0,t]$, we are led to consider the SDE, for $r \in [s,t]$,
\begin{equation}\label{stochflow:d+1}
	\begin{dcases}
	d_r
	\begin{pmatrix}
		\Phi^{(t)}_{r,s}(x,\xi) \\
		\Xi^{(t)}_{r,s}(x,\xi)
	\end{pmatrix}
	=
	\begin{pmatrix}
		B^{(t)}(r, \Phi^{(t)}_{r,s}(x,\xi)) \\
		C^{(t)}(r, \Phi^{(t)}_{r,s}(x,\xi)) \Xi^{(t)}_{r,s}(x,\xi)
	\end{pmatrix}
	dr
	+
	\Sigma^{(t)}(r, \Phi^{(t)}_{r,s}(x,\xi), \Xi^{(t)}_{r,s}(x,\xi)) d  W_r,  \\
	\begin{pmatrix}
		\Phi^{(t)}_{s,s}(x,\xi) \\
		\Xi^{(t)}_{s,s}(x,\xi)
	\end{pmatrix}
	=
	\begin{pmatrix}
		x \\
		\xi
	\end{pmatrix}.
	\end{dcases}
\end{equation}
It\^o's formula, \eqref{E:G}, and \eqref{stochflow:d+1} then yield that, for any $(s,x,\xi) \in [0,t) \times \RR^d \times \RR$,
\[
	r \mapsto G^{(t)}(r, \Phi^{(t)}_{r,s}(x,\xi), \Xi^{(t)}_{r,s}(x,\xi))
\]
is a martingale on $[s,t]$ with respect to the filtration $(\mcl F_r)_{r\in [0,t]}$ generated by the Wiener process $W$, and so, for all $r \in [s,t]$,
\begin{equation}\label{Gidentity}
	\EE \left[ G^{(t)}(r, \Phi^{(t)}_{r,s}(x,\xi), \Xi^{(t)}_{r,s}(x,\xi)) \mid \mcl F_s \right] = G^{(t)}(s, x,\xi).
\end{equation}
Observe that $\Phi^{(t)}_{r,s}$ is independent of $\xi$, while $\Xi^{(t)}_{r,s}$ can be written as $\Xi^{(t)}_{r,s}(x,\xi) = J^{(t)}_{r,s}(x) \xi$ for some scalar quantity $J^{(t)}_{r,s}(x)$, and so \eqref{stochflow:d+1} reduces to the two SDEs
\begin{equation}\label{SDE:2MCE}
	\begin{dcases}
	d_r \Phi^{(t)}_{r,s}(x) = -\left[ b(t-r, \Phi^{(t)}_{r,s}(x)) - (\sigma \cdot \nabla)\sigma^T(t-r, \Phi^{(t)}_{r,s}(x)) \right] dt + \sigma(t-r, \Phi^{(t)}_{r,s}(x)) dW_r, \quad r \in [s,t], \\
	 \Phi^{(t)}_{s,s}(x) = x
	 \end{dcases}
\end{equation}
and
\begin{equation}\label{SDE:JMCE}
	\left\{
	\begin{split}
	&d_r J^{(t)}_{r,s}(x) = \left[ - \div b + \tr[(\sigma \cdot \nabla)(\nabla \cdot \sigma)] + \frac{1}{2} |\div \sigma|^2 + \frac{1}{2} \tr[ \nabla \sigma \nabla\sigma^T] \right](t-r, \Phi^{(t)}_{r,s}(x)) J^{(t)}_{r,s}(x) dr \\
	&\qquad+ \div \sigma(t-r,\Phi^{(t)}_{r,s}(x)) J^{(t)}_{r,s}(x) d W_r, \quad r \in [s,t], \\
	&J^{(t)}_{s,s}(x) = 1.
	\end{split}
	\right.
\end{equation}
Standard but tedious computations involving It\^o's formula reveal that $J^{(t)}_{r,s}(x) = \det \nabla_x \Phi^{(t)}_{r,s}(x)$.

Taking $r = t$ and $\xi = 1$ in \eqref{Gidentity}, we thus arrive at
\[
	\EE\left[ f_0(\Phi^{(t)}_{t,s}(x)) J^{(t)}_{t,s}(x) \mid \mcl F_s \right] = g(s, x),
\]
and so, because $g(0,x) = f(t,x)$, we obtain the representation for solutions of \eqref{E:2MCE}:
\begin{equation}\label{formula:2MCE}
	f(t,x) = \EE\left[ f_0(\Phi^{(t)}_{t,0}(x)) J^{(t)}_{t,0}(x) \right].
\end{equation}
Let us note that $\Phi^{(t)}_{t,0}$ has the same law as $(\Phi_{t,0})^{-1}$, where $\Phi_{t,s}$ is the stochastic flow from \eqref{stochflow:general}. We can see this by duality with the nonconservative equation. Indeed, if $u$ is the solution of \eqref{E:2MTE} with $u(t,\cdot) = g$ for some given $g$, then
\[
	\int f_0(x) u(0,x)dx = \int f(t,x) g(x)dx.
\]
On the other hand, by \eqref{formula:2MTE} and \eqref{formula:2MCE},
\[
	\int f_0(x) u(0,x)dx = \EE \int f_0(x) g( \Phi_{t,0}(x))dx
\]
and
\[
	\int f(t,x) g(x)dx = \EE \int f_0(\Phi^{(t)}_{t,0}(x))  g(x) J^{(t)}_{t,0}(x) dx,
\]
so, using the change of variables formula and the fact that $f_0$ is arbitrary, we have $\EE[ g( \Phi_{t,0}(x)) ] = \EE[ g( [ \Phi^{(t)}_{t,0} ]^{-1}(x))]$ for all $g: \RR^d \to \RR$ and $x \in \RR^d$.

We now note that the SDE \eqref{SDE:2MCE} falls under the assumptions of Lemma \ref{L:backwardsstochflow}, and therefore, for every $0 \le s < t \le T$, there exists a unique solution $\Phi^{(t)}_{\cdot,s}$ with the properties laid out by that result. However, the main difficulty is that we do not know whether $\Phi^{(t)}_{t,0}$ is Lipschitz continuous on $\RR^d$ (see Remark \ref{R:Lipstoch?}). This prevents us from bounding $J^{(t)}_{t,0}$ uniformly in $L^\oo$ and passing to weak distributional limits. This is a major obstacle in using the formula \eqref{formula:2MCE} to identify the unique limiting distributional solution of \eqref{E:2MCE}, as we did for the first order equation \eqref{E:MCE}.

The exception is when $\sigma$ is independent of $x$. In that case, \eqref{SDE:2MCE} and \eqref{SDE:JMCE} become
\begin{equation}\label{SDE:2MCEconst}
	d_r \Phi^{(t)}_{r,s}(x) = - b(t-r, \Phi^{(t)}_{r,s}(x)) dr + \sigma(t-r) dW_r, \quad r \in [s,t], \quad
	 \Phi^{(t)}_{s,s}(x) = x
\end{equation}
and
\begin{equation}\label{SDE:JMCEconst}
	\del_r J^{(t)}_{r,s}(x) = -\div b(t-r, \Phi^{(t)}_{r,s}(x))J^{(t)}_{r,s}(x), \quad r \in [s,t], \quad J^{(t)}_{s,s}(x) = 1.
\end{equation}
The SDE \eqref{SDE:JMCEconst} is in fact an ODE with random coefficients. In particular, $J^{(t)}_{\cdot,s}$ has a deterministic bound.

We then characterize uniquely the limiting distributional solution of
\begin{equation}\label{E:2MCEconst}
	\del_t f - \nabla^2 \cdot (a(t)f) + \div( b(t,x) f) = 0 \quad \text{in } (0,T) \times \RR^d, \quad f(0,\cdot) = f_0.
\end{equation}

\begin{theorem}\label{T:2MCE}
	For $1 \le p < \oo$, the formula \eqref{formula:2MCE}, where $\Phi^{(t)}_{\cdot,s}$ and $J^{(t)}_{\cdot,s}$ are specified by respectively \eqref{SDE:2MCEconst} and \eqref{SDE:JMCEconst}, extends continuously to any $f_0 \in L^p(\RR^d)$. If $f_0 \in L^p(\RR^d)$ and $(b^\eps)_{\eps > 0}$ are as in \eqref{bregs} and $f^\eps$ is the corresponding solution of \eqref{E:2MCEconst}, then, as $\eps \to 0$, $f^\eps$ converges weakly in $C([0,T], L^p_\w(\RR^d))$ to $f$. If $f_0 \ge 0$, then there exists a unique nonnegative distributional solution of \eqref{E:2MCEconst}, which is given by \eqref{formula:2MCE}.
\end{theorem}

\begin{proof}
	Let $(b^\eps)_{\eps > 0}$ and $f^\eps$ be as in the statement of the theorem, and assume $f_0 \in C^2_c(\RR^d)$. Let $u^\eps$ be the solution of \eqref{E:2MTEconst} with velocity $b^\eps$ and with terminal data $u(t,\cdot) = g \in C^2_c(\RR^d)$ for some fixed $t \in [0,T]$. Then integration by parts yields
	\[
		\int f^\eps(t,x)g(x)dx = \int f_0(x) u^\eps(0,x)dx.
	\]
	By Theorem \ref{T:2MTEconst}, as $\eps \to 0$, $u^\eps$ converges strongly in $L^{p'}(\RR^d)$ to the function $u$ defined uniquely by $u(s,x) = g(\Phi_{t,s}(x))$. Therefore, any $C([0,T], L^p_\w(\RR^d))$-weak limit $f$ of $f^\eps$ as $\eps \to 0$ must satisfy
	\[
		\int f(t,x) g(x) dx = \int f_0(x) u(0,x)dx,
	\]
	and it follows that there is a unique such limiting function $f$.
	
	On the other hand, for $\eps > 0$,
	\[
		f^\eps(t,x) = \EE\left[ f_0\left( \Phi^{(t),\eps}_{t,0}(x) \right) J^{(t),\eps}_{t,0}(x) \right],
	\]
	where $\Phi^{(t),\eps}_{\cdot,s}$ and $J^{(t),\eps}_{\cdot,s}$ are as in respectively \eqref{SDE:2MCEconst} and \eqref{SDE:JMCEconst} with $b$ replaced everywhere by $b^\eps$. For fixed $t \in [0,T]$, uniformly in $\eps$, $\Phi^{(t),\eps}_{t,0}$ is Lipschitz continuous on $\RR^d$, and so $J^{(t),\eps}_{t,0} = \det \nabla_x \Phi^{(t),\eps}_{t,0}$ is bounded in $L^\oo$. By exactly the same arguments as in Lemma \ref{L:backwardsJ} and Theorem \ref{T:MCE}, we see that, as $\eps \to 0$, $\EE f_0\left( \Phi^{(t),\eps}_{t,0} \right) J^{(t),\eps}_{t,0}$ converges weakly in $L^p$ to $ \EE f_0\left( \Phi^{(t)}_{t,0} \right) J^{(t)}_{t,0} $. It follows that $f$ must be given by \eqref{formula:2MCE}. The fact that the formula extends to arbitrary $f_0 \in L^p(\RR^d)$ now follows from the a priori $L^p$ bounds in Theorem \ref{T:2MCEexist}.
	
	The uniqueness of nonnegative distributional solutions is then a consequence of the uniqueness of the forward flow established in Theorem \ref{T:fwdstochflowconst}, as well as the generalization of superposition to second-order Fokker-Planck equations (see Figalli \cite[Lemma 2.3]{Fig}).
\end{proof}

\bibliography{transportbib}{}
\bibliographystyle{acm}

\end{document}